\documentclass[11pt]{amsart}
\usepackage{latexsym,amssymb,amsfonts,amsmath,graphicx,fullpage,url,verbatim}
\usepackage[vcentermath,enableskew]{youngtab}
\usepackage{color}
\usepackage[all]{xy}
\usepackage{verbatim}
\usepackage{subfigure}
\usepackage{paralist}
\usepackage{hyperref}

\newcommand{\intequiv}{\overset{\mathrm{int}}{\equiv}}

\newtheorem{theorem}{Theorem}[section]
\newtheorem{proposition}[theorem]{Proposition}
\newtheorem{lemma}[theorem]{Lemma}
\newtheorem{corollary}[theorem]{Corollary}

\newtheorem*{theorem*}{Theorem}
\newtheorem*{theorem1*}{Theorem \ref{thm:main}}
\newtheorem*{theorem2*}{Theorem \ref{thm2}}
\newtheorem*{theorem3*}{Theorem \ref{thm3}}
\newtheorem*{cor1*}{Corollary \ref{cor1}}

\theoremstyle{definition}
\newtheorem{definition}[theorem]{Definition}
\newtheorem{example}[theorem]{Example}

\newtheorem{remark}[theorem]{Remark}
\numberwithin{equation}{section}
\def\F{\mathcal{F}}
\def\Cat{{\rm Cat}}
\def\relvol{{\rm relvol}}
\def\vol{{\rm vol}}
\def\CRY{{\mathcal{CRY}}}
\def\ASM{{\mathcal{ASMCRY}(n)}}
\def\F{{\mathcal{F}}}
\def\O{{\mathcal{O}}}
\def\I{{\mathcal{I}}}
\def\P{{\mathcal{P}}}
\def\ini{{\rm in }}
\def\fin{{\rm fin}}
\def\R{\mathbb{R}}
\def\v{{\rm v}}
\def\aff{{\rm aff}}
\def\ifl{{\rm ifl}}

\def\integer{{\rm int}}

\newcommand{\ee}{\end{equation}}

\newcommand{\bd}{\begin{definition}}
\newcommand{\ed}{\end{definition}}
\newcommand{\bt}{\begin{theorem}}
\newcommand{\et}{\end{theorem}}
\newcommand{\bl}{\begin{lemma}}
\newcommand{\el}{\end{lemma}}
\newcommand{\bp}{\begin{proposition}}
\newcommand{\ep}{\end{proposition}}
\newcommand{\bc}{\begin{corollary}}
\newcommand{\ec}{\end{corollary}}

\def\i{{\it in}}
\def\out{{\it out}}

\def\In{{\it In}}
\def\Out{{\it Out}}

\newcommand{\old}[1]{}

\author[K.\ M\'esz\'aros, A.H.\ Morales, J.\ Striker]{Karola M\'esz\'aros, Alejandro H.  Morales, Jessica Striker}

\thanks{KM was partially supported by a National Science
  Foundation Grant  (DMS 1501059) as well as by a von Neumann Fellowship at the IAS  funded by the Fund for Mathematics and the Friends of the Institute for Advanced Study.}
\thanks{AHM was partially supported
  by a CRM-ISM Postdoctoral Fellowship and an AMS-Simons travel grant.}
\thanks{JS was partially supported by a National Security Agency Grant (H98230-15-1-0041), the North Dakota EPSCoR National Science Foundation Grant (IIA-1355466), the NDSU Advance FORWARD program sponsored by National Science Foundation grant (HRD-0811239),  and a grant from the Simons Foundation/SFARI (527204, JS)}

\address{Department of Mathematics, Cornell University, Ithaca, NY 14853 and School of Mathematics, Institute for Advanced Study, Princeton, NJ 08540}
\address{Department of Mathematics and Statistics, UMass, Amherst, MA,
  01003}
\address{Department of Mathematics, North Dakota State University, Fargo, ND 58102}

\email{karola@math.cornell.edu, ahmorales@math.umass.edu, jessica.striker@ndsu.edu}

\title{On flow polytopes, order polytopes, and certain faces of the alternating sign matrix polytope}

\begin{document}

\begin{abstract}
In this paper we study an alternating sign matrix analogue of the Chan-Robbins-Yuen polytope, which we call the ASM-CRY polytope. We show that this polytope has Catalan many vertices and its volume is equal to the number of standard Young tableaux of staircase shape; we also determine its Ehrhart polynomial. We achieve the previous by proving that the members of a family of faces of the alternating sign matrix polytope which includes ASM-CRY are both order and flow polytopes. Inspired by the above results, we relate three established triangulations of order and flow polytopes, namely Stanley's triangulation of order polytopes, the Postnikov-Stanley triangulation of flow polytopes and the Danilov-Karzanov-Koshevoy triangulation of
flow polytopes.  We show that when a graph $G$ is a planar graph, in which case the flow polytope $\F_G$ is also an order polytope,  Stanley's triangulation of this order polytope is one of the Danilov-Karzanov-Koshevoy triangulations of $\F_G$. Moreover, for a general graph $G$ we  show that the set of Danilov-Karzanov-Koshevoy triangulations of $\F_G$ equal the set of framed Postnikov-Stanley triangulations of $\F_G$. We also describe explicit bijections between the combinatorial objects labeling the simplices in  the above triangulations.
\end{abstract}

 \maketitle

\setcounter{tocdepth}{1}
\tableofcontents 

\section{Introduction}
\label{sec:intro}

In this paper we study a family of faces of the alternating sign
matrix polytope inspired by an intriguing face of the Birkhoff
polytope: the Chan-Robbins-Yuen (CRY) polytope~\cite{CRY}. We call
these faces the ASM-CRY family of polytopes. Interest in the
CRY polytope centers around its volume formula as a product of
consecutive Catalan numbers; this has been proved~\cite{Z} via an
identity equivalent to the Selberg integral, but the problem of finding a \emph{combinatorial proof} remains open. We prove that the polytopes in the ASM-CRY family are \emph{order polytopes} and use Stanley's theory of order polytopes~\cite{Stop} to give a \emph{combinatorial proof} of
formulas for their volumes and Ehrhart polynomials. We also show that
these polytopes, and all order polytopes of strongly planar posets,
are \emph{flow polytopes}. of planar graphs (Theorem~\ref{o-f}). The
converse of this statement is due to Postnikov \cite{P13}
(private communication) and we here include a proof
(Theorem~\ref{f-o}). These  observations bring us to the general  question of relating the different known triangulations of  flow and order  polytopes. We show that when  $G$ is a planar graph,
in which case  the flow polytope of $G$ is also an order polytope, then Stanley's canonical triangulation of this order polytope \cite{Stop} is one of the \emph{Danilov-Karzanov-Koshevoy triangulations} of the flow polytope of $G$  \cite{kosh}, a statement first observed by Postnikov \cite{P13}. Moreover, for general $G$ we  show that the set of Danilov-Karzanov-Koshevoy triangulations of the flow polytope of $G$   equals  the set of \emph{framed Postnikov-Stanley triangulations} of the flow polytope of $G$ \cite{P13,S}. We also describe explicit bijections between the combinatorial objects labeling the simplices in  the above triangulations, answering a question posed by Postnikov \cite{P13}.

\medskip

We  highlight the main results of the paper in the following theorems.  While we define some of the notation here, some only  appears in later sections to which we give pointers after the relevant statements.

\medskip

In Definition~\ref{faces}, we define the \textit{ASM-CRY family} $\F(ASM)(n)$ of
polytopes $\mathcal{P}_{\lambda}(n)$ indexed by partitions
$\lambda\subseteq \delta_n$ where $\delta_n:=(n-1,n-2,\ldots,1)$. In
Theorem~\ref{prop:asmfaces}, we prove that the polytopes in this family are \textit{faces of the
  alternating sign matrix polytope} $\mathcal{A}(n)$ defined in \cite{BK,StrASM}.
In the case when $\lambda=\varnothing$  we obtain an analogue of the Chan-Robbins-Yuen (CRY)
polytope, which we call the ASM-CRY polytope, denoted by $\ASM$. This
polytope contains the CRY polytope.
Our main theorem about the family of polytopes $\F(ASM)(n)$ is the following. For
the necessary definitions, see Sections~\ref{sec:orderisflow} and \ref{sec:cryasm}.

\begin{theorem} \label{thm:main}  
  The polytopes in the family $\F(ASM)(n)$ are integrally equivalent to
  flow and order polytopes. In particular, $\mathcal{P}_\lambda(n)$ is
  integrally equivalent to  the order polytope of the  poset $(\delta_n\setminus\lambda)^*$ and the flow polytope $\F_{G_{(\delta_n\setminus\lambda)^*}}$. 
\end{theorem}

By Stanley's theory of order polytopes \cite{Stop} it follows that the
volume of the polytope $P_{\lambda}(n)$ for any $P_{\lambda}(n)\in
\F(ASM)(n)$ is given by the number of {\em linear extensions}  of the
poset $(\delta_n \setminus \lambda)^*$ (which equals the number of \emph{standard Young tableaux}
of skew shape $\delta_n/\lambda)$, and its Ehrhart polynomial in the
variable $t$ is
given by the {\em order polynomial} of the poset $(\delta_n \setminus \lambda)^*$. See Corollary~\ref{cor:ASMcor} for the general statement. We give the application to $\ASM$ in the corollary below. For further examples of polytopes in $\F(ASM)(n)$, see Figure~\ref{fig:family}.

\begin{corollary} \label{cor1}
$\ASM$ is integrally equivalent to the order polytope of the  poset $\delta_n^*$.
Thus, $\ASM$ has  $\Cat(n) = \frac{1}{n+1}\binom{2n}{n}$
vertices, its normalized volume is given by  \[\vol(\ASM) = \#
SYT(\delta_n),\] and its Ehrhart polynomial is 
\begin{equation} \label{eq:proctor}
L_{\ASM}(t) = \Omega_{\delta_n^*}(t+1) =
\prod_{1\leq i<j\leq n} \frac{2t+i+j-1}{i+j-1}.
\end{equation}
\end{corollary}

Also, since the CRY polytope is contained in the ASM-CRY polytope then
the formulas above are upper bound for the volume and number of
lattice points of the former polytope (Corollary~\ref{cor:boundscry}).

In Theorems~\ref{f-o} and \ref{o-f}, we make explicit the relationship
between flow and order polytopes, showing that they correspond under
certain planarity conditions of the respective graph and poset. As an
application we obtain flow polytopes with volume equal to the number
of standard Young tableaux of any skew shape $\lambda/\mu$ (see Figure~\ref{fig:tab2flow}).

For the definitions of  $(\delta_n\setminus\lambda)^*$ and $G_{(\delta_n\setminus\lambda)^*}$, see Definition \ref{star} and the discussion before Theorem \ref{o-f}, respectively.

As mentioned earlier, a canonical triangulation of order polytopes was given by Stanley \cite{Stop}, and two families of triangulations of flow polytopes were constructed by Postnikov and Stanley \cite{P13,S} as well as Danilov, Karzanov and Koshevoy \cite{kosh}. It is natural to understand the relation among these triangulations, and we prove the following results, the first of which was first observed by Postnikov \cite{P13}.  For the necessary definitions, see Sections \ref{sec:planar} and \ref{sec:tri}.

\bt[Postnikov \cite{P13}] \label{thm2}
Given a planar graph $G$, the canonical triangulation of the order polytope $ \widehat{\O}(P_G)$ maps to the  Danilov-Karzanov-Koshevoy triangulation of $\F_G$ coming from the planar framing via the integral equivalence map $\phi: \widehat{\O}(P_G)\rightarrow \F_G$ given in Theorem \ref{o-f}. 
\et

\bt \label{thm3} Given a framed graph $G$, the set of Danilov-Karzanov-Koshevoy triangulations of the flow polytope $\F_G$ equals  the set of framed Postnikov-Stanley triangulations of  $\F_G$.
\et

All three of the above-mentioned triangulations are indexed by natural sets of combinatorial objects and we  give explicit  bijections between these sets in Sections \ref{sec:planar} and \ref{sec:tri}.

\medskip

\textbf{The outline of the paper is as follows.} In Section \ref{sec:foa}, we discuss the Birkhoff and alternating sign matrix polytopes, as well as some of their faces. In Sections \ref{sec:flow} and \ref{sec:proofsf-o-and-o-f} we give background information on flow and order polytopes and show that flow polytopes of planar graphs are order polytopes and that order polytopes of strongly planar posets are flow polytopes. In Section \ref{sec:cryasm} we study a family of faces of the alternating sign matrix polytopes and show that they are integrally equivalent to both flow and order polytopes and calculate their volumes and Ehrhart polynomials in particularly nice cases. In Section \ref{sec:planar}, we study triangulations of flow polytopes of planar graphs (which include the polytopes of Section \ref{sec:cryasm}) and show that their canonical triangulations  defined  by Stanley \cite{Stop} are  also Danilov-Karzanov-Koshevoy triangulations \cite{kosh}. Finally, in Section \ref{sec:tri}, we study the Danilov-Karzanov-Koshevoy triangulations and the  framed Postnikov-Stanley triangulations   of  flow polytopes of an arbitrary graph. We show that these sets are equal. We also exhibit explicit bijections between the combinatorial objects indexing the various triangulations, answering a question raised by Postnikov \cite{P13}.

\section{Faces of the Birkhoff and alternating sign matrix polytopes}
\label{sec:foa}

In this section,  we explain the motivation for our study of certain faces
of the alternating sign matrix polytope.  We review some standard
facts of lattice point enumeration of integral polytopes \cite{BR},\cite[\S  4.6]{EC}. Given an integral polytope
$\mathcal{P}$, we denote by $\relvol(\mathcal{P})$ the volume of
$\mathcal{P}$ relative to its lattice and by $L_{\mathcal{P}}(t)$ the Ehrhart function
that counts the number of lattice points of the dilated polytope
$t\cdot \mathcal{P}$. A well known result of Ehrhart \cite{Eh} states
that if $\mathcal{P}$ is integral, then $L_{\mathcal{P}}(t)$ is a
polynomial of degree $\dim(\mathcal{P})$ with leading coefficient
$\relvol(\mathcal{P})$ (see \cite[Cor. 3.16]{BR}). The quantity
$\dim(\mathcal{P})!\cdot \relvol(\mathcal{P})$ is an integer (see
\cite[Cor. 3.17]{BR}) called the normalized volume that we denote by $\vol(\mathcal{P})$.

We say that two integral polytopes $\mathcal{P}$ in $\mathbb{R}^d$  and $\mathcal{Q}$ in $\mathbb{R}^m$ are
{\bf integrally equivalent}, which we denote by $\mathcal{P}
\intequiv \mathcal{Q}$, if there is an affine transformation
$\varphi:\mathbb{R}^d \to \mathbb{R}^m$ whose restriction to
$\mathcal{P}$ is a bijection $\varphi: \mathcal{P} \to \mathcal{Q}$
that preserves the lattice, i.e. $\varphi$ is a
 bijection between $\mathbb{Z}^d \cap \aff(\mathcal{P})$ and
 $\mathbb{Z}^m \cap \aff(\mathcal{Q})$ where $\aff(\cdot)$ denotes
the affine span. The map $\varphi$ is then an {\bf integral equivalence}. Note that integrally equivalent polytopes have the
same Ehrhart polynomials and therefore they have the same volume.   We remark that isomorphism and  unimodular equivalence are other terms sometimes used in the literature for what we will refer to as integral equivalence. 

Next we define the Birkhoff and
Chan-Robbins-Yuen polytopes; we then define the alternating sign
matrix counterparts.

\begin{definition}
The \textbf{Birkhoff polytope}, $\mathcal{B}(n)$, is defined as
\[\mathcal{B}(n) := {\Big \{} (b_{ij})_{i,j=1}^n \in \mathbb{R}^{n^2}
\mid b_{ij}\geq 0, \,\,\, { \sum_{i}} b_{ij} = 1, \,\,\, \sum_{j}
b_{ij} =1 {\Big \}}.\]
\end{definition}

Matrices in $\mathcal{B}(n)$ are called \textbf{doubly-stochastic matrices}. A well-known theorem of Birkhoff \cite{Birkhoff} and von Neumann \cite{VonNeumann} states that $\mathcal{B}(n)$, as defined above, equals the convex hull of the $n\times n$ permutation matrices. Note that $\mathcal{B}(n)$ has $n^2$ facets and dimension $(n-1)^2$, its vertices are the
permutation matrices, and its volume has been calculated up to $n=10$ by Beck and Pixton \cite{BP}.  
De Loera, Liu and Yoshida \cite{LLY} gave  a closed summation formula for the volume of $\mathcal{B}(n)$, which, while of interest on its own right, does not lend itself to easy computation. Shortly after, Canfield and McKay \cite{CM} gave an asymptotic formula for the volume. 

A special face of the Birkhoff polytope, first studied by Chan-Robbins-Yuen \cite{CRY}, is as follows.

\begin{definition} \rm
The \textbf{Chan-Robbins-Yuen polytope}, $\mathcal{CRY}(n)$, is defined as \[\mathcal{CRY}(n) := \left\{ (b_{{i}{j}})_{i,j=1}^n \in \mathcal{B}(n) \mid  b_{{i}{j}}=0 \mbox{ for }  i-j \geq 2\right\}.\]
\end{definition}

$\mathcal{CRY}(n)$ has dimension $\binom{n}{2}$ and $2^{n-1}$
vertices. This polytope was introduced by Chan-Robbins-Yuen \cite{CRY}
and in \cite{Z} Zeilberger calculated its normalized volume as the following product of \emph{Catalan numbers}.

\begin{theorem}[Zeilberger \cite{Z}] \label{thm:volCRY}
\[
\vol(\mathcal{CRY}(n)) = \prod_{i=1}^{n-2} \Cat(i)
\]
where $\Cat(i) = \frac{1}{i+1}\binom{2i}{i}$.
\end{theorem}

The proof in \cite{Z} used a relation (see Theorem~\ref{ps}) expressing
the volume as a value of the {\em Kostant
  partition function} (see Definition~\ref{def:KPF})  and a reformulation of the Morris
constant term identity \cite{WM} to calculate
this value. No combinatorial proof is known.

Next we give an analogue of the Birkhoff polytope in terms of
alternating sign matrices. Recall that \textbf{alternating sign matrices} (ASMs) \cite{ASM} are square matrices with the following properties:
\begin{itemize}
\item
entries $\in \{0,1,-1\}$,
\item
the entries in each row/column sum to 1, and
\item
the nonzero entries along each row/column alternate in sign.
\end{itemize}

The ASMs with no negative entries are the permutation matrices. See Figure \ref{fig:asm} for an example.

\begin{figure}[htbp]
\begin{center}
$\left( 
\begin{array}{rrr}
1 & 0 & 0 \\
0 & 1 & 0\\
0 & 0 & 1
\end{array} \right)
\left( 
\begin{array}{rrr}
1 & 0 & 0 \\
0 & 0 & 1\\
0 & 1 & 0
\end{array} \right)
\left( 
\begin{array}{rrr}
0 & 1 & 0 \\
1 & 0 & 0\\
0 & 0 & 1
\end{array} \right)
\left( 
\begin{array}{rrr}
0 & 1 & 0 \\
1 & -1 & 1\\
0 & 1 & 0
\end{array} \right)
\left( 
\begin{array}{rrr}
0 & 1 & 0 \\
0 & 0 & 1\\
1 & 0 & 0
\end{array} \right)
\left( 
\begin{array}{rrr}
0 & 0 & 1 \\
1 & 0 & 0\\
0 & 1 & 0
\end{array} \right)
\left( 
\begin{array}{rrr}
0 & 0 & 1 \\
0 & 1 & 0\\
1 & 0 & 0
\end{array} \right)$
\end{center}
\caption{All the $3\times 3$ alternating sign matrices.}
\label{fig:asm}
\end{figure}

\bd[Behrend-Knight \cite{BK}, Striker \cite{StrASM}] \label{def:asm}
The \textbf{alternating sign matrix polytope}, $\mathcal{A}(n)$, is defined as follows:
\[\mathcal{A}(n) :={\Big \{} (a_{ij})_{i,j=1}^n \in \mathbb{R}^{n^2} \mid 0\leq { \displaystyle\sum_{i=1}^{i'} a_{ij}\leq 1, \,\, 0\leq \displaystyle\sum_{j=1}^{j'} a_{ij}\leq 1, \,\,
\displaystyle\sum_{i=1}^n} a_{ij} = 1, \,\, \displaystyle\sum_{j=1}^n
a_{ij} =1 {\Big \}}, \]
where we have the first sum for any $1\leq i', j\leq n$, the second sum for any $1\leq j', i\leq n$,  the third sum for any $1\leq  j\leq n$, and the fourth sum for any $1\leq i\leq n$.
\ed

Behrend and Knight~\cite{BK}, and independently Striker~\cite{StrASM}, defined $\mathcal{A}(n)$.
The alternating sign matrix polytope can be seen as an analogue of the
Birkhoff polytope, since the former is the convex hull of all
alternating sign matrices (which include all permutation matrices)
while the latter is the convex hull of all permutation matrices. The
polytope $\mathcal{A}(n)$ has $4((n-2)^2+1)$ facets (for $n\geq 3$)~\cite{StrASM}, its dimension
is $(n-1)^2$, and its vertices are the $n\times n$ alternating sign
matrices~\cite{BK, StrASM}. The Ehrhart polynomial has been calculated
up to $n=5$ \cite{BK}. Its normalized volume for $n=1,\ldots,5$ is
calculated to be 
\[
1, 1, 4, 1376, 201675688,
\]
and no asymptotic formula for its volume is known.

In analogy with $\CRY(n)$, we study a special face of the ASM polytope we call the {ASM-CRY polytope} 
(and show, in Theorem~\ref{prop:asmfaces}, it is indeed a face of $\mathcal{A}(n)$).

\begin{definition} \label{def:asmcry}
The \textbf{ASM-CRY polytope} is defined as follows.
\[\ASM := \left\{ (a_{{i}{j}})_{i,j=1}^n \in \mathcal{A}(n) \mid  a_{{i}{j}}=0 \mbox{ for } {i}-{j}\geq 2\right\}.\]
\end{definition}

Since the $\CRY(n)$ polytope has a nice product formula for its
normalized volume, it is then natural to wonder if the volume of the
alternating sign matrix analogue of $\CRY(n)$, which we denote by
$\ASM$,  is similarly nice.  In Theorem \ref{thm:main} and Corollary~\ref{cor1}, we show that
$\ASM$ is both a flow and order polytope, and using the theory
established for the latter, we give the volume formula and the Ehrhart
polynomial of $\ASM$. Just like in the  $\CRY(n)$ case, all formulas
obtained are combinatorial. Unlike in the $\CRY(n)$ case, all the
proofs involved are combinatorial. In Theorem~\ref{thm:main}, we extend these results to a
family of faces $\F(ASM)(n)$ of the ASM polytope, of which $\ASM$ is a
member; see Section~\ref{sec:cryasm}. 

\begin{figure}
\subfigure[]{
\includegraphics[height=4cm]{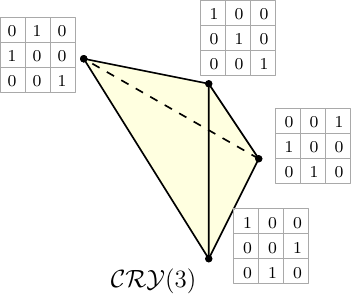}
}
\subfigure[]{
\includegraphics[height=4cm]{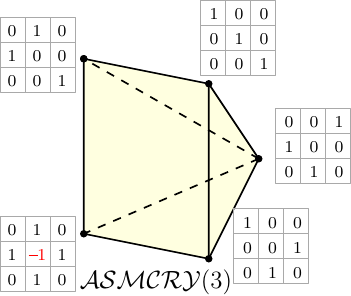}
}
\subfigure[]{
\includegraphics{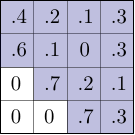}
}
\subfigure[]{
\includegraphics{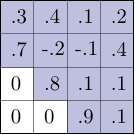}
}
\caption{(a) The
  polytope $\mathcal{CRY}(3)$ in $\mathbb{R}^3$, (b) the polytope
  $\mathcal{ASMCRY}(3)$ in $\mathbb{R}^3$, (c) a doubly-stochastic
  matrix in $\mathcal{CRY}(4)$, (d) a matrix in
  $\mathcal{ASMCRY}(4)$.}
\label{fig:intro}
\end{figure}

\section{Flow and order polytopes}
\label{sec:flow}
In order to state and prove Theorem~\ref{thm:main} in Section~\ref{sec:cryasm}, we need to discuss flow and order polytopes.
In Section~\ref{subsec:31}, we define flow and order polytopes and also explain how to see  $\CRY(n)$ as the flow polytope of the complete graph. In Sections~\ref{sec:flowisorder} and~\ref{sec:orderisflow}, we state  in Theorems \ref{f-o} and \ref{o-f} that the flow polytope of a planar graph is the order polytope of a related poset, and vice versa. We give the proofs of these theorems in Section \ref{sec:proofsf-o-and-o-f}.

\subsection{Background and definitions}
\label{subsec:31}

Let $G$ be a connected graph on the vertex set $[n]:=\{1,2, \ldots,
n\}$ with edges directed from the smaller to larger vertex. 
Denote by $\ini(e)$ the smaller (initial) vertex of edge $e$ and $\fin(e)$ the bigger (final) vertex of edge $e$.

\begin{definition} \label{def:flow}
Given a vector ${\bf a} = (a_1,a_2,\ldots,a_{n-1},-\sum_{i=1}^{n-1}
a_i)$ with $a_i \in \mathbb{Z}_{\geq 0}$, a \textbf{flow} $fl$ on $G$ with
{\bf netflow} ${\bf a}$ is a function $fl: E(G) \rightarrow \R_{\geq 0}$  such that 
 for $i=1,2,\ldots,n-1$ 
  $$\sum_{e \in E, \ini(e)=i}fl(e)  \,-\, \sum_{e \in E, \fin(e)=i}fl(e)= a_i$$
and 
$$\sum_{e \in E,  \fin(e)=n}fl(e) = \sum_{i=1}^{n-1} a_i.$$
\medskip
The \textbf{flow polytope} $\F_G({\bf a})$ associated to the graph $G$
and netflow vector ${\bf a}$ is the set of all flows $fl: E(G)
\rightarrow \R_{\geq 0}$ on $G$ with
netflow ${\bf a}$. We denote the set of integer flows of $\F_G({\bf
  a})$ by $\F_G^{\integer}({\bf a})$.
\end{definition}

\begin{definition}
A flow $fl$ of size one on $G$ is a flow on $G$ with netflow
$(1,0,\ldots,0,-1)$.  That is 
\[
\sum_{e \in E, \ini(e)=1} fl(e) \,=\, \sum_{e\in E, \fin(e)=i} fl(e) \,=\, 1,
\]
and for $2\leq i \leq n-1$
 $$\sum_{e \in E, \fin(e)=i}fl(e)\,=\, \sum_{e \in E, \ini(e)=i}fl(e).$$
  \medskip
  The \textbf{flow polytope} $\F_G$ associated to the graph $G$ is the set of all flows $fl: E(G) \rightarrow \R_{\geq 0}$ of size one on $G$. 
  \end{definition}

We assume that in our flow polytopes $\F_G$ each vertex $v \in \{2,3,\ldots,n-1\}$ in $G$  has both
incoming and outgoing edges.  
Note that this restriction  is not a serious
one. If there is a vertex  $v \in [2,n-1]$  with only incoming or
outgoing edges, then in $\F_G$ the flow on all these edges must be
zero, and thus, up to removing such vertices, any flow polytope $\F_G$
is integrally equivalent to a flow polytope defined as above.

The  polytope $\F_G$ is a convex polytope in the Euclidean space
  $\R^{\#E(G)}$ and its dimension is  $\dim(\F_G)=\#E(G)-\#V(G)+1$
  (e.g. see \cite{BV2}). The vertices of $\F_G$ are characterized as
  follows.

\begin{proposition}[{\cite[Cor. 3.1]{gallo}}] \label{prop:verticesflowpoly}
Let $G$ be a connected graph with vertices $[n]$ with edges oriented
from smaller to bigger vertices. Then the vertices of $\F_G$ are the
unit flows on maximal directed paths or {\bf routes} from the source $1$ to the sink $n$.
\end{proposition}

 Figure ~\ref{5} shows the equations of $\F_{K_5}$ and
  explains why this polytope is  integrally equivalent to $\CRY(4)$. The same
  correspondence shows that $\F_{K_{n+1}}$ and $\CRY(n)$ coincide. The
  following theorem connects volumes of flow polytopes and Kostant partition functions.  
  
\begin{theorem}[Postnikov-Stanley \cite{P13,S}, Baldoni-Vergne \cite{BV2}] \label{ps}
For a loopless graph $G$ on the   vertex set  $\{1,2\ldots,n\}$, with $d_i=indeg_i(G)-1$,
\[
\vol\left(\F_{G} \right)=K_{G}(0,d_2, \ldots, d_{n-1},
-\sum_{i=2}^{n-1} d_i),
\]  
where $K_{G}({\bf a})$ is the {Kostant partition function}, $indeg_i(G)$ denotes the indegree of vertex $i$ in $G$  and $\vol$ is normalized volume.
\end{theorem}

Recall the definition of the Kostant partition function.

\begin{definition} \label{def:KPF} The  \textbf{Kostant partition function}   $K_G(\v)$ is
the number of ways to write the vector $\v$ as a nonnegative linear
combination of the positive type $A_{n-1}$ roots corresponding to the
edges of $G$, without regard to order. The edge  $(i, j)$, $i<j$,
of $G$ corresponds  to the vector $e_i-e_j$, where $e_i$ is the
$i^{th}$ standard basis vector in $\mathbb{R}^{n}$.  
\end{definition}

It is easy to see by definition that the number of integer flows on
$G$ with netflow ${\bf a}$, that is, the size of $\F_G^{\integer}({\bf
  a})$ or number of integer points in the flow polytope $\F_G({\bf a})$, equals  $K_G({\bf a})$. In particular, 
the Ehrhart polynomial of $\F_G$ in variable $t$ is equal to $K_G(t, 0, \ldots,0, -t).$

\begin{figure}
\begin{center}
\includegraphics[scale=.7]{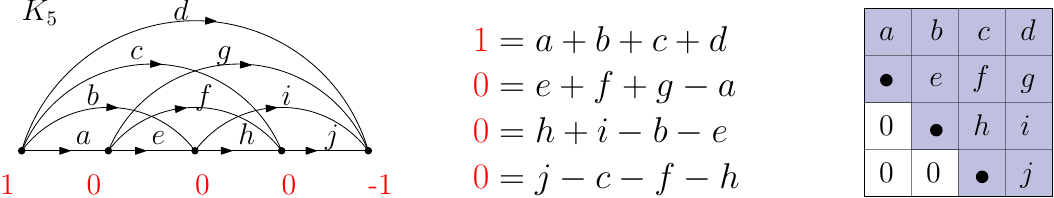}
 \caption{Graph $K_5$ with edges directed from smaller to bigger
   vertex. The flow variables on the edges are $a, b, c, d, e, f, g,
   h, i, j$, the net flows in the vertices are $1,0,0,0,-1$. The equations defining the flow polytope corresponding to $K_5$ are in the middle. Note that these same equations define $\CRY(4)$ as can be seen from the matrix on the left, where we denoted by $\bullet$ entries that are determined by the variables  $a, b,\ldots,j$. }
 \label{5}
 \end{center}
\end{figure}

\medskip

Now we are ready to define order polytopes and relate them to flow polytopes.

\begin{definition}[Stanley \cite{Stop}]
The \textbf{order polytope}, $\mathcal{O}(P)$, of a poset $P$ with
elements $\{t_1,t_2,\ldots,t_n\}$  is the set of points
$(x_1,x_2,\ldots,x_n)$ in $\mathbb{R}^n$ with $0\leq x_i\leq 1$ and if
$t_i \leq_P t_j$ then $x_i \leq x_j$. We identify each point
$(x_1,x_2,\ldots,x_n)$ of  $\mathcal{O}(P)$ with the function $f:P\to
\mathbb{R}$ with $f(t_i)=x_i$. 
\end{definition}

In our proofs, we will often use the polytope $\hat{\mathcal{O}}(P)$,
which is 
integrally equivalent to $\mathcal{O}(P)$ \cite[Sec. 1]{Stop}:

\begin{definition}[Stanley \cite{Stop}] \label{def:stan}
 Let $\widehat{P}$ be the poset obtained from $P$ by adjoining a minimum element $\hat{0}$ and a maximum
element $\hat{1}$. Define a polytope $\widehat{\mathcal{O}}(P)$ to be the set of functions $g:\widehat{P}\to
\mathbb{R}$ satisfying $g(\hat{0}) = 0$, $g(\hat{1}) = 1$, and $g(x) \leq g(y)$ if $x\leq y$ in $\widehat{P}$.
\end{definition}

\begin{lemma}[Stanley \cite{Stop}] \label{lem:stan}
The map $\nu: \widehat{\mathcal{O}}(P) \to \mathcal{O}(P)$ given by  $(g(x))_{x\in\widehat{P}}\mapsto (g(x))_{x\in {P}}$ is an integral equivalence. \end{lemma}

In general, computing or finding a combinatorial interpretation for the volume of a polytope is a hard problem. Order polytopes are an especially nice class of polytopes whose volume has a combinatorial interpretation.

\begin{theorem}[Stanley \cite{Stop}] \label{thm:stop}
Given a poset $P$ we have that 
\begin{compactitem}
\item[(i)] the vertices of $\mathcal{O}(P)$  are in bijection with
  characteristic functions of complements of order ideals of $P$, 
\item[(ii)]  the normalized volume of $\mathcal{O}(P)$  is $e(P)$, where
  $e(P)$ is the number of linear extensions of $P$,
\item[(iii)] the Ehrhart polynomial $L_{\mathcal{O}(P)}(m)$ of
  $\mathcal{O}(P)$ equals the order polynomial $\Omega(P,m+1)$ of $P$.
\end{compactitem}
\end{theorem}

\begin{definition}
Given a poset $P$ and a positive integer $m$, the \textbf{order polynomial} $\Omega(P,m)$ is the number of order preserving maps $\eta:P\rightarrow \left\{1,2,\ldots,m\right\}$.
\end{definition}

\subsection{Flow polytopes of planar graphs are order polytopes}
\label{sec:flowisorder}
The following theorem, which states that a flow polytope of a planar
graph is an order polytope, is a result communicated to us by
Postnikov \cite{P13}.  Given a connected
graph $G$ with the conventions of Section~\ref{subsec:31}, we say that $G$ is
{\bf planar} if it has a planar embedding so that if vertex $i$ is
in position $(x_i,y_i)$ then $x_i < x_j$ whenever $i<j$. We denote by $G^*$ the \textbf{truncated dual} graph of
$G$, which is the dual graph with the vertex corresponding to the
infinite face deleted. The orientation of the edges of $G$ induces an
orientation of the edges of $G^*$ (faces of $G$)
from lower to higher $y$-coordinates of the end points. This allows us to consider $G^*$ as the Hasse
diagram of a poset that we denote by $P_G$. See
Figure~\ref{fig:planar_inverse}. Note that by Euler's formula,
$\#P_G=\#E(G)-\#V(G)+1$ which equals $\dim(\mathcal{F}_G)$. Let
$\widehat{P}_G := P_G \sqcup \{\hat{0},\hat{1}\}$.

\begin{figure}
\begin{center}
\includegraphics[scale=0.8]{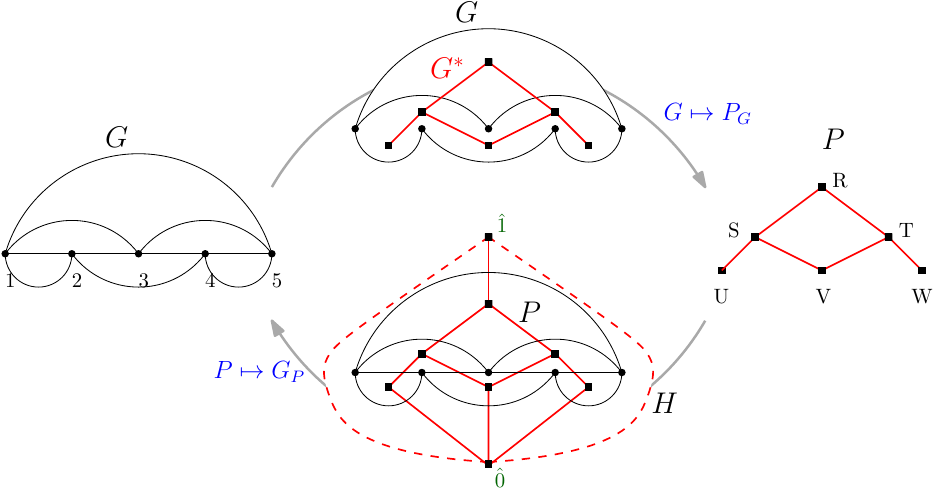}
\caption{Illustration of how to obtain the Hasse diagram $P_G$ from a
  planar graph $G$ (top arrow), and how to obtain a planar graph $G_P$
  from a strongly planar poset $P$ (bottom arrow).}
\label{fig:planar_inverse}
\end{center}
\end{figure}

\begin{theorem}[Postnikov \cite{P13}] \label{f-o} Let $G$ be a planar
  graph on the vertex set $[n]$ such that at each vertex $v \in
  [2,n-1]$ there are both incoming and outgoing edges. Fix a planar
  embedding of $G$ with the above conventions. Then the map $\varphi: \mathcal{F}_G \to \widehat{\mathcal{O}}(P_G)$  given in  Definition \ref{def:varphi} is an integral equivalence. In particular, $ \mathcal{F}_G \intequiv \widehat{\mathcal{O}}(P_G)\intequiv \mathcal{O}(P_G)$.
  \end{theorem}

\begin{definition} \label{def:varphi}
Define  $\varphi: \mathcal{F}_G \to \widehat{\mathcal{O}}(P_G)$ by  $fl \mapsto (f(x))_{x\in P_G}$, where $f: \widehat{P}_G \to
\mathbb{R}_{\geq 0}$ is given by \begin{equation} \label{def:eqflow}
f(x)=\sum_{e  \in \mathsf{p}} fl(e).
\end{equation}
The latter  sum is taken over the edges $e$
  that are intersected by a(ny) path $\mathsf{p}$ in $\widehat{P}_G$
  from $\hat{0}$ to $x$.   
  \end{definition}

\begin{figure}
\begin{center}
\includegraphics[scale=0.8]{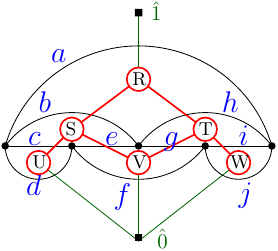}
\caption{Illustration of the maps $fl \mapsto f$ and $f \mapsto fl$  from Definitions \ref{def:varphi} and \ref{def:phi} explained in Examples \ref{ex:mapsflow2order1} and  \ref{ex:mapsflow2order2}.}
\label{fig:mapflow2order}
\end{center}
\end{figure}

\begin{example} \label{ex:mapsflow2order1}
Given the graph $G$ and the corresponding poset $P_G$ in
Figure~\ref{fig:mapflow2order}, the   map  $fl \mapsto f$ from Definition \ref{def:varphi} is as follows:
 
\begin{minipage}[c]{0.4\linewidth}
\begin{align*}
f(R) &= fl(b)+fl(c)+fl(d),\\
f(S) &= fl(c) + fl(d)\\
f(T) & = fl(i) + fl(j)\\
f(U) & = fl(d)\\
f(V) &= fl(f)\\
f(W) & = fl(j).
\end{align*}
\end{minipage}

\end{example}

 Lemma \ref{lem:map2op_well_defined} below 
 shows that $f$ in Definition \ref{def:varphi} is well-defined, while Lemma \ref{lem:map} shows that $\varphi$  indeed maps points in $\F_G$ to
points in $\widehat{\mathcal{O}}({P}_G)$. 
The proof of Theorem \ref{f-o} is given in  Section~\ref{sec:proofsf-o-and-o-f}.

\subsection{Order polytopes of strongly planar posets are flow polytopes}
\label{sec:orderisflow}
We now state a converse of  Theorem \ref{f-o},  showing that the order
polytope of a strongly  planar poset   is a flow polytope.   A poset
$P$ is   {\bf strongly planar} if the Hasse diagram of $\widehat{P}:=P
\sqcup \{\hat{0}, \hat{1}\}$ has a planar embedding with $y$
coordinates respecting the order of the poset. For example, the
``bowtie'' poset defined by the relations $a<c, a<d, b<c, b<d$ is planar, 
but not strongly planar. Given a strongly planar poset $P$, let $H$ be
the (planar) graph obtained from the Hasse diagram of $\widehat{P}$ with two additional edges from $\widehat{0}$ to $\widehat{1}$, one of which
goes to the left of all the poset elements and another to the
right. We can then define the graph $G_P$ to be the truncated dual of
$H$. 
The orientation of  $G_P$  is inherited from the poset in the
following way: if in the construction of the truncated dual, the edge
$e$ of $G_P$ crosses the edge $x\rightarrow y$ where $x<y$ in $P$,
then $y$ is on the left and $x$ is on the right as you traverse $e$. See
Figure~\ref{fig:planar_inverse}.

\begin{theorem}\label{o-f} If $P$ is a strongly planar poset, then
 the map $\phi: \widehat{\mathcal{O}}(P) \to \mathcal{F}_{G_P}$  given in  Definition \ref{def:phi} is an integral equivalence. In particular, $\mathcal{O}(P) \intequiv   \widehat{\mathcal{O}}(P)\intequiv   \mathcal{F}_{G_P}$.
\end{theorem}

\begin{definition} \label{def:phi}
Define  $\phi: \widehat{\mathcal{O}}(P) \to \mathcal{F}_{G_P}$ by  $(f(x))_{x\in \hat{P}} \mapsto fl$, where $fl: E(G_P) \to
\mathbb{R}_{\geq 0}$ is given by 
\begin{equation} \label{def:eq-order}
fl(e)=f(y)-f(x),
\end{equation}
where  $e$ crosses the
Hasse diagram edge $x\rightarrow y$ (in the dual construction). 
\end{definition}

\begin{example} \label{ex:mapsflow2order2}
Given the graph $G$ and the corresponding poset $P_G$ in
Figure~\ref{fig:mapflow2order}, the  map  $f \mapsto fl$ from Definition \ref{def:phi} is as follows:

\begin{minipage}[c]{0.25\linewidth}
\begin{align*}
fl(a) &= 1-f(R)\\
fl(b) &= f(R)  - f(S)\\
fl(c) & = f(S)-f(U)\\
fl(d) & = f(U)\\
fl(e) &= f(S)-f(V)\\
&\\
\end{align*}
\end{minipage}
\begin{minipage}[c]{0.25\linewidth}
\begin{align*}
fl(f) & = f(V)\\
fl(g) &= f(T)-f(V)\\
fl(h) &=f(R)-f(T)\\
fl(i) & =f(T)-f(W)\\
fl(j) &=f(W).\\
& \\
\end{align*}
\end{minipage}
\end{example}

We postpone the proof of Theorem \ref{o-f} result to Section \ref{sec:proofsf-o-and-o-f}.

\section{Proofs of Theorem~\ref{f-o} and Theorem~\ref{o-f}} \label{sec:proofsf-o-and-o-f}

This section provides the proofs of  Theorems~\ref{f-o} and ~\ref{o-f}.
 
\begin{lemma} \label{lem:map2op_well_defined}
Given a flow $fl \in \F_G$, the map $f:\widehat{P}_G \to \mathbb{R}_{\geq 0}$
is independent on the path $\mathsf{p}$ chosen in \eqref{def:eqflow}. 
\end{lemma}

\begin{proof}
Let $\mathsf{p}_1$ and $\mathsf{p}_2$ be two paths in $\widehat{P}_G$ from $\hat{0}$ to $x$. We show that 
\begin{equation} \label{pp}
\sum_{e \in \mathsf{p}_1} fl(e) = \sum_{e \in \mathsf{p}_2} fl(e).
\end{equation}
If $\mathsf{p}_1$ and $\mathsf{p}_2$ coincide, \eqref{pp}  is trivial.  We induct on    the number of
vertices of $G$ enclosed by the two paths $\mathsf{p}_1$ and $\mathsf{p}_2$.
Without loss of generality, assume $\mathsf{p}_1$
is left of $\mathsf{p}_2$ given the planar drawing of $G$.    Let $v$ be
the  vertex with the smallest $x$-coordinate enclosed by the two paths $\mathsf{p}_1$ and $\mathsf{p}_2$ in the planar drawing  of $G$. By construction, all
the incoming edges in $G$ to $v$ are crossed by path
$\mathsf{p}_1$ and $x$ is not a face between two incoming edges to
$v$. Next, we do the following local move to change the path
$\mathsf{p}_1$: let $\mathsf{p}'_1$ be the path that coincides with
$\mathsf{p}_1$ except that it crosses the outgoing edges of $v$ (see Figure~\ref{fig:localmovepaths}). By
conservation of flow on vertex $v$, the sum of the flow of the
incoming edges to $v$ equals the sum of the flow of the outgoing edges
from $v$. Since these are the only crossed edges that $\mathsf{p}_1$ and
$\mathsf{p}'_1$ differ on we have that 
\[
\sum_{e \in \mathsf{p}_1} fl(e) = \sum_{e \in \mathsf{p}'_1} fl(e).
\]
The paths $\mathsf{p}'_1$ and $\mathsf{p}_2$ have one fewer vertex of
$G$ enclosed by them than the paths $\mathsf{p}_1$ and $\mathsf{p}_2$. By
induction we have 
\[
 \sum_{e \in \mathsf{p}'_1} fl(e) = \sum_{e \in \mathsf{p}_2} fl(e).
\]
Comparing the latter two equations, the result follows.
\end{proof}

\begin{figure}
\includegraphics{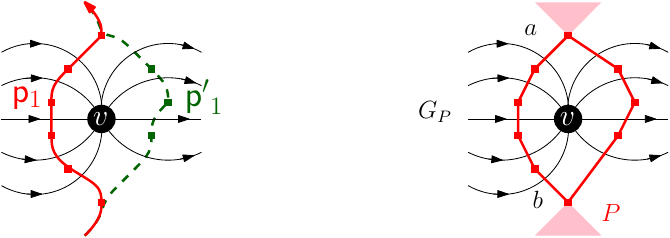}
\caption{Left: Local move of paths in the proof of
  Lemma~\ref{lem:map2op_well_defined}. Right: Illustration of why flow
is conserved in the map from $\widehat{\mathcal{O}}(P)$ to $\mathcal{F}_{G_P}$.}
\label{fig:localmovepaths}
\end{figure}

Next, we show that given a flow $fl$ in $\F_G$ the point $\varphi(fl)=(f(x))_{x\in P_G}$ is in $\widehat{\mathcal{O}}(P_G)$. 

\begin{lemma} \label{lem:map}
Given a flow $fl \in \F_G$, the image $\varphi(fl) \in \widehat{\mathcal{O}}(P_G)$.
\end{lemma}

\begin{proof} Note  that $f(\hat{0})=0$. 
 We have that $0\leq f(x)$ since $fl(e)\geq 0$ for all edges
$e$. Also, $f(x)\leq 1$ since the set of edges whose sum of flows
equals $f(x)$ can always be extended to a path from $\hat{0}$ to
$\hat{1}$. By repeated application of Lemma~\ref{lem:map2op_well_defined}, the total flow in such a path is $1$. Thus, $f(\hat{1})=1$.  Next, if $x'$ covers $x$ in $\widehat{P}_G$
then there is an edge $e'$ in $G$ separating the graph faces $x$ and
$x'$. Thus $f(x')=fl(e')+f(x)\geq f(x)$. Hence the linear map $f$ takes a point $(fl(e))_{e \in E(G)}$ of $\F_G$ to the
point $(f(x))_{x \in P_G}$ of  the order polytope $\widehat{\mathcal{O}}(P_G)$. 
\end{proof}

\begin{lemma} \label{lem:map1}
Given a point in $\widehat{\mathcal{O}}(P)$ viewed as a
function $f: \widehat{P} \to \mathbb{R}_{\geq 0}$, the flow $fl : E(G_P) \to
\mathbb{R}_{\geq 0}$ as in Definition \ref{def:phi} is in $\mathcal{F}_{G_P}$.
\end{lemma}

\begin{proof}
Let $f:\widehat{P}\to \mathbb{R}_{\geq 0}$ be a point in $\widehat{\mathcal{O}}(P)$ and
let $e$ be an edge in $G_P$ crossing the Hasse diagram edge $x
\rightarrow y$ of $\widehat{P}$. Since $x \leq_P y$ then by definition
of $\widehat{\mathcal{O}}(P)$, $fl(e)=f(y)-f(x)
\geq 0$. 

Next, we find the netflows at each vertex of $G$. Consider the
leftmost (rightmost) path in $\widehat{P}$ from $\hat{0}$ to $\hat{1}$. This path
crosses all the outgoing (incoming) edges in $G$ of vertex $1$ (vertex
$n$). We have that 
\[
\sum_{ e \in E, \ini(e)=1} fl(e) = \sum_{e \in E,  \fin(e)=n}fl(e)  = f(\hat{1})-f(\hat{0}) = 1-0 =1.
\]
For an internal vertex $v\in [2,n-1]$, let $a$ be the  face
bordering the highest incoming and outgoing edge to
$v$. Similarly, let $b$ be the  face
bordering the lowest incoming and outgoing edge to
$v$. Consider the paths $\mathsf{p}_{in}$ and $\mathsf{p}_{out}$ be
the paths in $\widehat{P}$ from $b$ to $a$ crossing the incoming and
outgoing edges to $v$ respectively (see
Figure~\ref{fig:localmovepaths}, Right). Then the total incoming and
outgoing flow to vertex $v$ are
\begin{align*}
\sum_{e \in E, \fin(e)=v}fl(e) & = \sum_{z \to w \text{ in }
  \mathsf{p}_{in}} (f(w)-f(z)) = f(a) - f(b),\\
\sum_{e \in E, \ini(e)=v}fl(e) & = \sum_{z \to w \text{ in }
  \mathsf{p}_{out}} (f(w)-f(z)) = f(a) - f(b),
\end{align*}
This shows the flow is conserved on vertex $v$, and thus $fl(\cdot)$
is in $\mathcal{F}_{G_P}$.
\end{proof}

Lemma \ref{lem:map1}  shows that $\phi(\widehat{\mathcal{O}}(P)) \subset \mathcal{F}_{G_P}$.

\begin{proof}[Proof of Theorem~\ref{f-o} and Theorem~\ref{o-f}]
Note that given a planar graph $G$ we have that $Q:=P_G$ is a strongly
planar poset and that $G_{Q} = G$. Given a flow $fl$  in
$\F_G$, let $f=\varphi(fl)$ the corresponding point in
$\widehat{O}(P_G)$ and $fl'$ be the flow $\phi(\varphi(fl))=\phi(f)$. Let $e$ be an edge of $G$ crossing the
Hasse diagram edge $x\to y$ in $\widehat{P}_G$
\begin{align*}
fl'(e) &= f(y) - f(x)\\
&= \sum_{e_1 \in \mathsf{p}} fl(e_1) - \sum_{e_2 \in
  \mathsf{q}} fl(e_2), 
\end{align*}
where $\mathsf{p}$ is a path $\widehat{P}_G$ from $\hat{0}$ to $y$ and
$\mathsf{q}$ is a path $\widehat{P}_G$ from $\hat{0}$ to $x$. By
Lemma~\ref{lem:map2op_well_defined} the value of $f(y)$ is independent
of the choice of path, so by letting $\mathsf{p} = \mathsf{q}+ x\to y$
the last difference becomes $fl(e)$, showing that $fl'(e)=fl(e)$. A similar argument shows that
$\varphi \circ \phi$ is the identity. Thus the maps $\phi$ and
$\varphi$ are inverses of each other and they both preserve integer points. Therefore, $\phi$ and
$\varphi$ are integral equivalences.   Using Lemma \ref{lem:stan} giving $\widehat{\mathcal{O}}(P)\intequiv  \mathcal{O}(P)$ for any poset $P$ we are done.
\end{proof}

\begin{remark}
By Theorem~\ref{f-o}, if $G$ is a planar graph then $\F_G$ is integrally
equivalent to an order polytope. This raises the question of whether
this relation holds for non-planar graphs: for instance for the
polytope $\CRY(n) \cong \F_{K_{n+1}}$ for $n\geq 4$. We
can use a similar construction to that in the proof of Theorem~\ref{f-o}  to show that
$\F_{K_5}$ and $\F_{K_6}$ are integrally equivalent to the order polytopes of the posets:
\begin{center}
\includegraphics[scale=1.2]{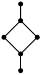}, \quad 
\includegraphics[scale=1.2]{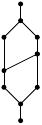}
\end{center}
We leave it as a question whether $\F_{K_7}$ (dimension $15$, $32$ vertices, volume $140$)  is an order polytope.
\end{remark}

\begin{figure}
\includegraphics[scale=0.8]{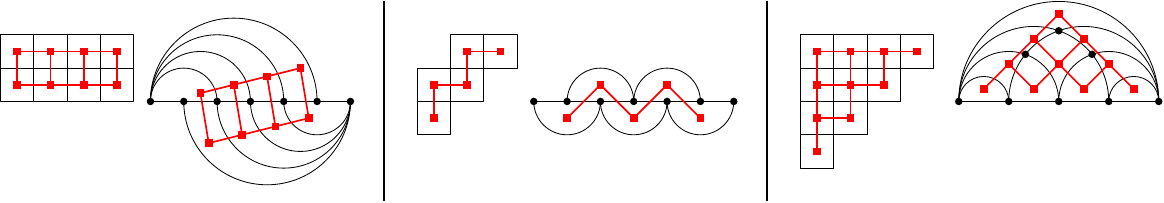}
\caption{Examples of Young diagrams, their associated planar posets $P$ and
  graphs $G_P$ such that the order polytope $\mathcal{O}(P)$ and
  $\F_{G_P}$ are integrally equivalent.}
\label{fig:tab2flow}
\end{figure}

\section{$\ASM$ and the family of polytopes $\F(ASM)$}
\label{sec:cryasm}

In this section, we introduce the ASM-CRY family of polytopes $\F(ASM)$, which includes
$\ASM$, and show that each of these polytopes is a face of the ASM polytope. We, furthermore, show that each polytope in this family is both  an order
and a flow polytope. Then, using the theory of order     polytopes as discussed in Section~\ref{subsec:31}, we determine their volumes and Ehrhart polynomials.

\bd \label{faces}
Let $\delta_n=(n-1, n-2, \ldots, 2, 1)$ be the staircase partition
considered as the positions $(i,j)$ of an $n\times n$ matrix given by
$\left\{(i,j) \mid {j}-{i}\geq 1\right\}$. Let the partition
$\lambda=(\lambda_1,\lambda_2,\ldots,\lambda_{k})\subseteq \delta_n$
denote matrix positions $\left\{(i,j) \mid 1\leq i \leq k, \,\,
  n-\lambda_i+1\leq j\leq n, \,\, \lambda_i\leq n-i\right\}$.

We define the {\bf ASM-CRY family}
\[\F(ASM)(n) := \left\{  \mathcal{P}_{\lambda}(n) \mid \lambda \subseteq\delta_n \right\},\]
where
\[ \mathcal{P}_{\lambda}(n):=\left\{(a_{{i}{j}})_{i,j=1}^n \in \mathcal{A}(n) \mid  a_{{i}{j}}=0 \mbox{ for } {i}-{j}\geq 2 \mbox{ and for } (i,j)\in\lambda \right \}.\]
\ed

Note that $\mathcal{P}_\varnothing(n)=\ASM$, as in Definition \ref{def:asmcry}. 
\medskip

In the following proposition we give a convex hull description of the polytopes in this family.
\begin{proposition} \label{prop:asmconv}
The polytope $\P_\lambda(n) \in \F(ASM)(n)$ is the convex hull of the $n\times n$ alternating sign matrices $(A_{ij})_{i,j=1}^n$ with $A_{{i}{j}}=0 \mbox{ for } {i}-{j}\geq 2 \mbox{ and for } (i,j)\in\lambda$.
\end{proposition}

\begin{proof}
Let $\mathcal{Q}_{\lambda}(n)$ denote the convex hull of the $n\times n$ alternating sign matrices $(A_{ij})_{i,j=1}^n$ with $A_{{i}{j}}=0 \mbox{ for } {i}-{j}\geq 2 \mbox{ and for } (i,j)\in\lambda$.
It is easy to see that $\mathcal{Q}_{\lambda}(n)$ is contained in $\mathcal{P}_{\lambda}(n)$, since matrices in both polytopes have the same prescribed zeros and satisfy the inequality description of the full ASM polytope $\mathcal{A}(n)$. 

It remains to prove that $\mathcal{P}_{\lambda}(n)$ is contained in
$\mathcal{Q}_{\lambda}(n)$. Suppose there exists a matrix
$b=(b_{ij})_{i,j=1}^n\in\P_{\lambda}(n)$ such that
$b\notin\mathcal{Q}_{\lambda}(n)$. We know that $b$ is in the convex
hull of all $n\times n$ ASMs. So $b=\mu_1 A^1 + \mu_2 A^2 + \cdots +
\mu_{k} A^k$, where $A^1,\ldots A^k$ are distinct $n\times n$ alternating sign matrices and
$\mu_1,\ldots,\mu_k > 0$.  At least one of these ASMs, say $A^1$ must
have a nonzero entry $A^1_{ij}$ for some $(i,j)$ satisfying either
${i}-{j}\geq 2 \mbox{ or  } (i,j)\in\lambda$. Suppose ${i}-{j}\geq 2$;
the argument follows similarly in the case $(i,j)\in\lambda$. Now
since $b_{ij}=0$ and $A^1_{ij}\neq 0$, there must be another ASM, say
$A^2$ such that $A^2_{ij}$ is nonzero of opposite sign. Say
$A^1_{ij}=1$ and $A^2_{ij}=-1$. Then by the definition of an
alternating sign matrix, there must be $j'<j$ such that
$A^2_{ij'}=1$. But $b_{ij'}=0$ as well, so there must be an $A^3$ with $A^3_{ij'}=-1$ and $j''<j'$ such that $A^3_{ij''}=1$. Eventually, we will reach the border of the matrix and reach a contradiction.  Thus, $\mathcal{P}_{\lambda}(n) = \mathcal{Q}_{\lambda}(n)$.
\end{proof}

We show in Theorem~\ref{prop:asmfaces} below that the polytopes in $\F(ASM)(n)$ are faces of $\mathcal{A}(n)$. First, we need some terminology from~\cite{StrASM}. 
Consider $n^2 + 4n$ vertices on a square grid: $n^2$
`internal' vertices $(i, j)$
and $4n$ `boundary' vertices $(i, 0)$, $(0, j)$, $(i, n + 1)$, and $(n + 1, j)$, where $1\leq i,j\leq n$. Fix the orientation of this grid so that the first coordinate increases from top to bottom and the second coordinate increases from left to right, as in a matrix.
The \textbf{complete flow grid} $C_n$ is defined as the directed graph on these vertices
with directed edges pointing in both directions between neighboring internal vertices within the grid, and also directed edges from internal vertices to neighboring border vertices. That is, $C_n$ has edge set $\{((i, j),(i, j \pm 1)),((i, j),(i \pm 1, j)) \mid i, j = 1,\ldots, n\}$.
A \textbf{simple flow grid} of order $n$ is a subgraph of $C_n$ consisting of all the
vertices of $C_n$, and in which four edges are incident to each internal vertex: either four
edges directed inward, four edges directed outward, or two horizontal edges pointing in
the same direction and two vertical edges pointing in the same
direction. An \textbf{elementary flow grid} is a subgraph of $C_n$
whose edge set is the union of the edge sets of some simple flow grids. See Figure~\ref{fig:elem}.

\begin{theorem} \label{prop:asmfaces} The polytope $\P_\lambda(n) \in \F(ASM)(n)$ is a face of $\mathcal{A}(n)$, of dimension $\binom{n}{2}-|\lambda|$. In particular, 
$\P_{\varnothing}(n)=\ASM$ is a face of $\mathcal{A}(n)$, of dimension $\binom{n}{2}$.
\end{theorem}

\begin{proof} 
In Proposition 4.2 of~\cite{StrASM}, it was shown that the simple flow grids of order $n$ are in bijection with the $n\times n$ alternating sign matrices. In this bijection, the internal vertices of the simple flow grid correspond to the ASM entries; the sources correspond to the ones of the ASM, the sinks correspond to the negative ones, and all other vertex configurations correspond to zeros.
In Theorem 4.3 of~\cite{StrASM}, it was shown that the faces of $\mathcal{A}(n)$ are in bijection with $n\times n$ elementary flow grids, with the complete flow grid $C_n$ in bijection with the full ASM polytope $\mathcal{A}(n)$. This bijection was given by noting that the convex hull of the ASMs in bijection with all the simple flow grids contained in an elementary flow grid is, in fact, an intersection of facets of the ASM polytope $\mathcal{A}(n)$, and is thus a face of $\mathcal{A}(n)$. Since, by Proposition~\ref{prop:asmconv}, $\P_{\lambda}(n)$ equals the convex hull of the ASMs in it, we need only show there exists an elementary flow grid whose contained simple flow grids correspond exactly to these ASMs.

We can give this elementary flow grid explicitly. We claim that the
directed edge set $S:=\bigcup_{(i,j)} S_{i,j}$ where
\[
S_{i,j} := \begin{cases}
\left\{\left((i, j),(i, j-1)\right), \left((i, j),(i+1, j)\right)\right\}  &\text{ if } i-j \geq 2\\
 \left\{\left((i, j),(i, j+1)\right), \left((i, j),(i-1, j)\right)\right\}  &\text{ if } (i,j) \in \lambda\\
 \left\{\left((i, j),(i, j \pm 1)\right),\left((i, j),(i \pm 1, j)\right)\right\}  & \text{ otherwise}
\end{cases}
\]
 is the union of the directed edge sets of all the simple flow grids in bijection with ASMs in $\P_{\lambda}(n)$. It is clear that the directed edge set of any simple flow grid corresponding to an ASM in $\P_{\lambda}(n)$ is in $S$; it remains to show that any  edge in $S$ appears in some simple flow grid. 
Note that the directed edges listed in the first two cases appear in every simple flow grid in bijection with an ASM in $\P_{\lambda}(n)$. For the remaining edges, note that if $A_{ij}=1$, then in the corresponding simple flow grid, $S_{i,j}=\left\{\left((i, j),(i, j \pm 1)\right),\left((i, j),(i \pm 1, j)\right)\right\}$. It is easy to construct a permutation matrix $A\in \P_{\lambda}(n)$ with $A_{ij}=1$ for any fixed $(i,j)$ with $i-j<2$ and $(i,j)\notin\lambda$.
 Thus the digraph with the edge set $S$ is an elementary flow grid. Furthermore, no other simple flow grid can be constructed from directed edges in this set, since such a simple flow grid would have to include an edge pointing in the wrong direction in either the region ${i}-{j}\geq 2$ or $(i,j)\in\lambda$.
Thus, $\P_{\lambda}(n)$ is a face of $\mathcal{A}(n)$.

To calculate the dimension of $\P_{\lambda}(n)$, we use the following notion from \cite{StrASM}. A \textbf{doubly directed region} of an elementary flow grid is a connected
collection of cells in the grid completely bounded by double directed edges but containing no
double directed edges in the interior. Theorem 4.5 of~\cite{StrASM} states that the dimension of a face of $\mathcal{A}(n)$ equals the number of doubly directed
regions in the corresponding elementary flow grid. The number of doubly directed regions in the elementary flow grid corresponding to $\P_{\lambda}(n)$ equals $(n-1)^2-\left(\binom{n-1}{2}+|\lambda|\right)=\binom{n}{2}-|\lambda|$. See Figure~\ref{fig:elem}.
\end{proof}

\begin{figure}
\begin{center}
(a) \hspace{1.65in} (b) \hspace{1.65in} (c) \hspace{1.5in}

\includegraphics[height=4cm]{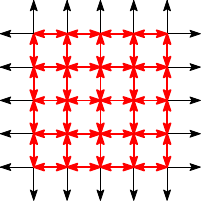}
\hspace{.3in}
\includegraphics[height=4cm]{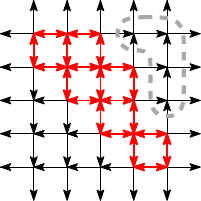}
\hspace{.3in}
\includegraphics[height=4cm]{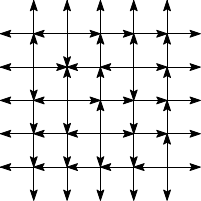}
\end{center}
\caption{(a) The complete flow grid $C_5$, which corresponds to the full ASM polytope $\mathcal{A}(5)$. (b) The elementary flow grid corresponding to $\P_{\lambda}(5)$ with $\lambda = (2,1,1)$. Note there are six doubly directed regions, thus $\P_{\lambda}(5)$ is a face of $\mathcal{A}(5)$ of dimension six. (c) A simple flow grid which corresponds to a $5\times 5$ ASM and is contained in the elementary flow grid of (b).}
\label{fig:elem}
\end{figure}

Our main result regarding  $\F(ASM)(n)$ is Theorem \ref{thm:main},
which we prove below. It requires the following
definition (see Figure~\ref{fig:family} for examples); also, recall from Section~\ref{sec:orderisflow} the definition of $G_P$.

\begin{definition} \label{star}
Let $\delta_n$ and $\lambda\subseteq\delta_n$ be as in Definition~\ref{faces}. Let $(\delta_n\setminus\lambda)^*$ be the poset with elements $p_{ij}$ corresponding to the positions $(i,j)\in \delta_n\setminus\lambda$ with partial order $p_{ij}\leq p_{i'j'}$ if $i\geq i'$ and $j\leq j'$.
\end{definition}

\smallskip

We now prove Theorem~\ref{thm:main} by first establishing two lemmas to show that
$\mathcal{P}_\lambda(n)$ is integrally equivalent to the order polytope
of the poset $(\delta_n\setminus\lambda)^*$. Then since this poset is
strongly planar, by Theorem~\ref{o-f} its order polytope is integrally equivalent
to the flow polytope $\F_{G_{(\delta_n\setminus\lambda)^*}}$.  

Given a matrix $\left(a_{ij}\right)_{i,j=1}^n\in \P_{\lambda}(n)$,
define the {\bf corner sum matrix} $\left(c_{ij}\right)_{i,j=1}^n$ by
\[c_{ij}=\displaystyle\sum_{\substack{1\leq i'\leq i,\\ j\leq j'\leq n}}
a_{i'j'}.\] 
For $S\subseteq \mathbb{R}$, let $\mathcal{A}(\delta_n \setminus \lambda,S)$ be the set of functions
$g: \delta_n \setminus \lambda \to S$. We view the order
polytope of $(\delta_n \setminus \lambda)^*$ as a subset of
$\mathcal{A}(\delta_n \setminus \lambda, [0,1])$. Define $\Psi:
\mathcal{P}_{\lambda}(n) \to \mathcal{A}(\delta_n \setminus \lambda, \mathbb{R})$
by $a \mapsto g_a$ where $g_a(i,j) = 1-c_{ij}$. See the second map in
Figure~\ref{fig:cryasmtoOP}.

\begin{lemma}\label{asm:lem1}
The image of $\Psi$ is in $\mathcal{A}(\delta_n
\setminus \lambda, [0,1])$, i.e. if $a\mapsto g_a$ then $g_a(i,j)=1-c_{ij} \in [0,1]$.
\end{lemma}

\begin{proof}
We first show that $c_{ij}\geq 0$ for all $i$ and $j$. By the defining inequalities of the ASM polytope $\mathcal{A}(n)$ (see Definition~\ref{def:asm}), we have that the partial row and column sums of any $a\in\mathcal{A}(n)$ satisfy the following for each fixed $1\leq i,j\leq n$: 
\begin{equation}
\label{eq:partialsums}
\sum_{i'=1}^i a_{i'j}\geq 0 \text{ and } \sum_{j'=j}^n a_{ij'}\geq 0.
\end{equation}

Since $c_{ij}=\sum_{i'=1}^i \sum_{j'=j}^n a_{i'j'}$ and the interior sum is nonnegative by (\ref{eq:partialsums}), $c_{ij}\geq 0$ as desired.

Next we show that for $a\in \P_{\lambda}(n)$, $c_{ij}\leq 1$ for all $j>i\geq 1$. (Note this is not true for all matrices in
$\mathcal{A}(n)$; for example the permutation matrix corresponding to
$4321$ has $c_{2 3}=2$.) 

First note $c_{1j}\leq 1$ for all $j$, since by (\ref{eq:partialsums}) each $a_{1j}\geq 0$ and $\sum_{j=1}^n a_{ij}=1$. 

Now fix $j>i\geq 2$. We have
\[c_{ij}=\sum_{j'=j}^n\sum_{i'=1}^i a_{i'j'} = \sum_{j'=1}^n\sum_{i'=1}^i a_{i'j'} - \sum_{j'=1}^{j-1}\sum_{i'=1}^i a_{i'j'} = i - \sum_{j'=1}^{j-1}\sum_{i'=1}^i a_{i'j'}\]
since the sum of each row is $1$. Note  
\[\sum_{j'=1}^{j-1}\sum_{i'=1}^i a_{i'j'} = \sum_{j'=1}^{i-1}\sum_{i'=1}^i a_{i'j'} + \sum_{j'=i}^{j-1}\sum_{i'=1}^i a_{i'j'}.\] Now \[\sum_{j'=1}^{i-1}\sum_{i'=1}^i a_{i'j'} = \sum_{j'=1}^{i-1}\sum_{i'=1}^n a_{i'j'} = \sum_{j'=1}^{i-1} 1 = i-1\] since $a_{i'j'}=0$ for $j'<i<i'$.
So \[c_{ij}=i-(i-1)-\sum_{j'=i}^{j-1}\sum_{i'=1}^i a_{i'j'}=1-\sum_{j'=i}^{j-1}\sum_{i'=1}^i a_{i'j'}.\]
But by (\ref{eq:partialsums}), $\sum_{i'=1}^i a_{i'j'}\geq 0$, so
$\sum_{j'=i}^{j-1}\sum_{i'=1}^i a_{i'j'}$ and thus $c_{ij}\leq 1$ for
all $j>i$.

We therefore have that  $0\leq c_{ij}\leq 1$ for all $j>i$, so that
$0\leq g_a(i,j) \leq 1$ as desired.
\end{proof}

\begin{lemma} \label{asm:lem2}
The image of  $\Psi$ is in the order polytope
$\mathcal{O}\left((\delta_n\setminus\lambda)^*\right)$.
\end{lemma}

\begin{proof}
By Lemma~\ref{asm:lem1} we know that the image of $\Psi$ is in $\mathcal{A}(\delta_n
\setminus \lambda, [0,1])$.
Note that if $i'\leq i$ and $j'\geq j$, then
$c_{ij}\geq c_{i'j'}$, thus we have that $g_a(i,j)\leq g_a(i',j')$ if and only
if $(i,j)\leq (i',j')$ in $(\delta_n\setminus\lambda)^*$. 
So $g_a$ is in the order
polytope $\mathcal{O}\left((\delta_n\setminus\lambda)^*\right)$. 
\end{proof}

\begin{proof}[Proof of Theorem~\ref{thm:main}]
By Lemmas~\ref{asm:lem1} and \ref{asm:lem2} we have that the map
$\Psi$ is an affine transformation from $\mathcal{P}_{\lambda}(n)$ to
$\mathcal{O}\left((\delta_n\setminus\lambda)^*\right)$ of the form
$a \mapsto {\bf 1} - {\bf A}a$ where $A$ is
a $0,1$-upper unitriangular matrix. Thus, $\Psi$ is a bijection between
$\mathcal{P}_{\lambda}(n)$ and
$\mathcal{O}\left((\delta_n\setminus\lambda)^*\right)$ that preserves
their respective lattices. This shows that the two polytopes are
integrally equivalent.

Finally since the poset $(\delta_n\setminus \lambda)^*$ is strongly
planar, by Theorem \ref{o-f} $\P_{\lambda}(n)$ is also integrally
equivalent to the flow polytope
$\F_{G_{(\delta_n\setminus\lambda)^*}}$.  \end{proof}

By Stanley's theory of order polytopes \cite{Stop} (see Theorem~\ref{thm:stop}) we express the
volume and Ehrhart polynomial of the polytopes in this family in terms
of their associated posets. Recall that $e(P)$ denotes the number of
linear extensions of the poset $P$.

\begin{corollary}[\cite{Stop}]
\label{cor:ASMcor}
For $\mathcal{P}_\lambda(n)$ in $\F(ASM)(n)$ we have that its normalized
volume is
\[
\vol(\mathcal{P}_\lambda(n)) = e\left( (\delta_n\backslash \lambda)^*\right),
\]
and its Ehrhart polynomial is 
\[
L_{\mathcal{P}_\lambda(n)}(t) = \Omega_{(\delta_n\backslash \lambda)^*}(t+1).
\]
\end{corollary}

\begin{figure}
\begin{center}
\includegraphics[scale=.8]{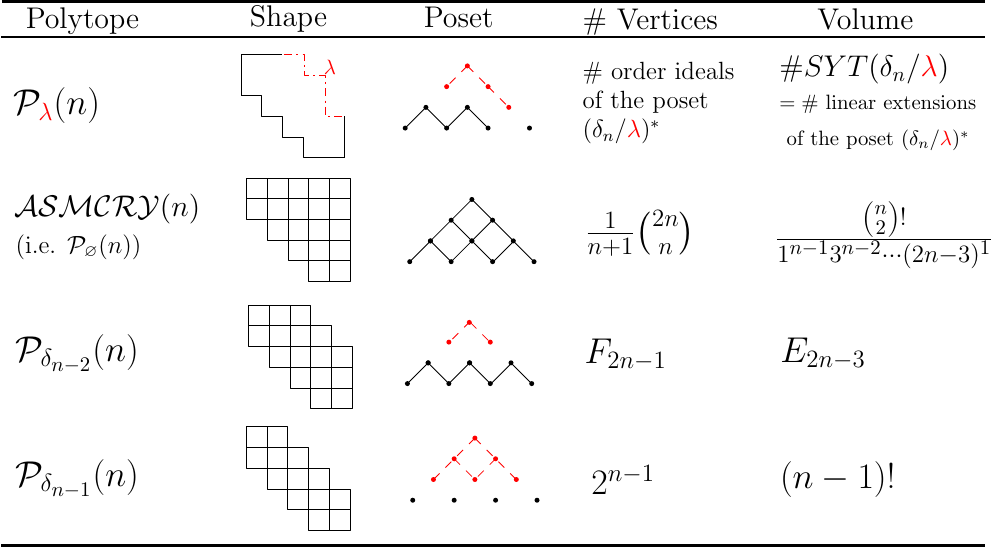}
\end{center}
\caption{Some polytopes in the family $\F(ASM)(n)$ and their corresponding numbers of vertices and volumes; see Theorem~\ref{thm:main} and Corollaries~\ref{cor1}, \ref{cor:ASMcor}, \ref{cor:anti}, and \ref{cor:fib}. `Shape' refers to the entries in the matrix not fixed to be zero.  All diagrams are drawn in the case $n=5$.}
\label{fig:family}
\end{figure}

Note that using Theorem \ref{ps} and the discussion below it, we can
express the volume and Ehrhart polynomial of any flow polytope as a
Kostant partition function. Thus, Theorem \ref{thm:main} gives us
several Kostant partition function identities. Corollaries \ref{cor1}, \ref{cor:anti} and \ref{cor:fib} compute the volumes and Ehrhart polynomials of three subfamilies of polytopes in $\F(ASM)(n)$ that are associated to posets with a nice number of linear extensions and
vertices. This includes the ASM-CRY polytope. See Figure~\ref{fig:family}.

\begin{figure}
\begin{center}
\includegraphics[scale=0.8]{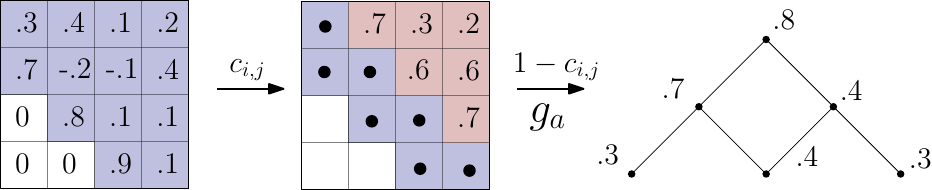}
\end{center}
\caption{A map from a point in $\mathcal{ASMCRY}(4)$ to a point in the order polytope. First, take the northeast corner sum of each entry above the main diagonal. Then subtract that value from 1.}
\label{fig:cryasmtoOP}
\end{figure}

\begin{proof}[Proof of Corollary~\ref{cor1}]
When
$\lambda=\varnothing$,  $\P_{\varnothing}(n)$ is integrally equivalent to the
order polytope $\mathcal{O}_{\delta_n^*}$ of the poset $\delta_n^*$ (that is, the type $A_{n-1}$
positive root lattice). 

By Theorem~\ref{thm:stop} the number of vertices and volume 
of $\P_{\varnothing}(n)$  are given by the number of order ideals and 
linear extensions of the poset  $\delta_n^*$ respectively. Next we
compute each of these. 

The order ideal of the poset $\delta_n^*$ correspond to shapes
$\lambda \subseteq \delta_n$ which in turn correspond to Dyck paths
counted by the Catalan number $C_n = \frac{1}{n+1}\binom{2n}{n}$.

The number of linear extension of this poset is
the number of standard Young tableaux (SYT) of shape
$\delta_n=(n-1,n-2,\ldots,2,1)$. Thus
\[
\vol \P_{\varnothing}(n) = \#SYT(\delta_{n}) =\frac{\binom{n}{2}!}{1^{n-1}3^{n-2}\cdots (2n-3)^1},
\]
where the second equality follows by using  the hook-length formula
\cite[Cor. 7.21.6]{EC} to compute this number of tableaux. 

Lastly, by Theorem~\ref{thm:stop} $L_{\P_{\varnothing}(n)}(t) = \Omega_{\delta^*_{n}}(t+1)$. When
$t$ is an integer, $\Omega_{\delta^*_{n}}(t+1)$ counts 
the the number of plane partitions of shape $\delta_n$
with largest part $\leq t$. By a result of Proctor \cite{Proctor} (see also
\cite{FK}) this number is given by the product formula in the RHS of \eqref{eq:proctor}.
\end{proof}

Since the polytope $\CRY(n)$ is contained in $\ASM$ then we can bound the volume and
number of lattice points of the former with the corresponding volume
and number of lattice points of the latter. 

\begin{corollary} \label{cor:boundscry}
For $n\geq 1$ and $t\in \mathbb{N}$ we have that 
\begin{align*}
\prod_{i=1}^{n-2} \Cat(i) & \,\leq\, \#
SYT(\delta_n)\\
L_{\CRY(n)}(t) & \,\leq \, \prod_{1 \leq i<j\leq n}
                 \frac{2t+i+j-1}{i+j-1}.
\end{align*}
\end{corollary} 

\begin{proof}
Since $\CRY(n) \subseteq \ASM$ and both polytopes have the same
dimension then we can compare their normalized volumes to obtain $\vol(\CRY(n)) \leq
\vol(\ASM)$. Also by comparing the number of lattice points of their
dilations we have that for $t$ in $\mathbb{N}$, $L_{\CRY(n)}(t) \leq
L_{\ASM}(t)$. The result then follows by combining these bounds with
Theorem~\ref{thm:volCRY} and Corollary~\ref{cor1} respectively.
\end{proof}

\begin{remark}
Since the normalized volume of $\CRY(n)$ has a product formula, one wonders if there is a product formula for the number of its integer points; however, data suggests the answer to be negative (see data in \cite[Sec. 6]{LLY}, \cite{moorefield} and \cite{flowsite}).  It would interesting
to study the asymptotics of $L_{\CRY(n)}(t)$.
\end{remark}

We give a few other examples of polytopes in the family $\F(ASM)(n)$ that
have known nice formulas for the volume, namely, in the cases $\lambda = \delta_{n-k}$ for $k\geq 1$.
See Figure~\ref{fig:family}. 

Let $[n]$ be the poset with $n$ elements and no
relations and $z_{2n-1}$ and $z_{2n}$ denote the {\bf zigzag posets} with
$2n-1$ and $2n$ elements, respectively: \includegraphics{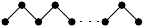} and \includegraphics{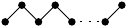}.

\begin{corollary}
\label{cor:anti}
$\P_{\delta_{n-1}}(n)$ is
  integrally equivalent to the order polytope $\O({[n-1]})$ of the antichain $[n-1]$, it has $2^{n-1}$ vertices and its normalized volume equals $(n-1)!$.
\end{corollary}

\begin{proof}
Since the poset $[n-1]$ is an antichain, there are no relations, so the number
  of order ideals is $2^{n-1}$ and the number of linear extensions is $(n-1)!$. Thus, the result follows from Theorem~\ref{thm:main}. 
\end{proof}

\begin{corollary}
\label{cor:fib} 
$\P_{\delta_{n-2}}(n)$ is
  integrally equivalent to the order polytope $\O({z_{2n-3}})$ of
  the zigzag poset $z_{2n-3}$, its number of vertices is given by the
  Fibonacci number $F_{2n-1}$, and its normalized volume is given by the Euler number $E_{2n-3}$. 
\end{corollary}

\begin{proof}
The  number of order ideals of the zigzag poset with $n$ elements is
given by the Fibonacci number $F_{n+2}$. To see this, note the posets $z_0$ and $z_1$ have $F_2=1$ and $F_3=2$ order ideals
respectively. For the zigzag $z_n$, the number of order ideals equals
the sum of order ideals of $z_{n-1}$ and $z_{n-2}$ depending on
whether or not the order ideals includes the leftmost (minimal) element of the
poset. The result follows by induction.

The
  number of linear extensions of this poset is the number
  of SYT of skew shape ${\delta_n/\delta_{n-2}}$ which is given by the Euler number $E_{2n-3}$. Thus, the result follows from Theorem~\ref{thm:main}. 
\end{proof}

\begin{remark}
For the case $\lambda=\delta_{n-k}$, the polytope
$\mathcal{P}_{\delta_{n-k}}(n)$ is integrally equivalent to the order polytope of
the poset $(\delta_n\setminus \delta_{n-k})^*$. The number of vertices
of the polytope (order ideals of the poset) is given by the number of Dyck paths with height at most $k$
\cite[\href{https://oeis.org/A211216}{A211216}]{OEIS}, \cite[\S
3.1]{KRT}. The volume of the polytope is given by the
number of skew SYT of shape $\delta_n/\delta_{n-k}$. There are
formulas for this number of SYT as determinants of Euler numbers (e.g
see Baryshnikov-Romik \cite{BR}). 
\end{remark}

We now turn from our investigation of the family of polytopes $\F(ASM)(n)$ to triangulations of flow and order polytopes.

\section{Triangulations of flow polytopes of planar graphs}
\label{sec:planar}
 
As we have seen in Section \ref{sec:flow}, flow polytopes of planar
graphs are integrally equivalent to order polytopes. In this section
we relate a known triangulation of flow polytopes by
Danilov--Karzanov--Koshevoy and a well known triangulation of order
polytopes.

\subsection{Canonical triangulation of order polytopes}
Recall that
vertices of an order polytope $\O(P)$ correspond to characteristic
functions of order filters (i.e. complements of order ideals). Stanley \cite{Stop} gave a canonical
way of triangulating the order polytope $\O(P)$ for an arbitrary poset
$P$. Namely, for a linear extension
$(a_1, a_2, \ldots, a_m)$ of the poset $P$ on elements $\{a_1, a_2,
\ldots, a_m\}$, define the simplex \begin{equation} \Delta_{a_1, a_2,
    \ldots, a_m}:=\{(x_1, \ldots, x_m) \in [0,1]^{m} \mid x_{a_1} \leq
  x_{a_2}\leq \cdots \leq x_{a_m}\}. \end{equation} 

Note that the $m+1$ vertices of this simplex are $0,1$ vectors whose
$0$-coordinates are indexed by length $k$ prefixes $a_1,\ldots,a_k$ of the linear
extension for $k=0,1,\ldots,m$.
The simplices  $\Delta_{a_1, a_2,
    \ldots, a_m}$
corresponding to all linear extensions of $P$ are top dimensional
simplices in a triangulation of $\O(P)$, which we refer to as the
\textbf{canonical triangulation of $\O(P)$}.  There are also two
established combinatorial  ways of triangulating  flow polytopes: one given by
Postnikov and Stanley (PS)  \cite{P13,S} (defined in Section \ref{sec:ps}), and one by Danilov, Karzanov and
Koshevoy (DKK) \cite{kosh} (defined in Section \ref{sec:kosh}). All the
aforementioned triangulations are unimodular. The goal of this section
is to relate the DKK triangulation of flow polytopes of planar graphs
and Stanley's linear extension triangulation of the corresponding
order polytope.

As before, it will be more convenient for us to work with the
integrally equivalent polytope $\widehat{\O}(P) \intequiv \O(P)$ and
the integral equivalence $\nu$ from Lemma \ref{lem:stan}. The canonical triangulation of $\O(P)$ maps under $\nu^{-1}$ to the  canonical triangulation of $\widehat{\O}(P)$. We will denote $\nu^{-1}(\Delta_{a_1, a_2,
    \ldots, a_m})$ by  $\widehat{\Delta}_{a_1, a_2,
    \ldots, a_m}$. Of course:
    
     \begin{equation} \widehat{\Delta}_{a_1, a_2,
    \ldots, a_m}:=\{(x_{\hat{0}}, x_1, \ldots, x_n, x_{\hat{1}}) \in [0,1]^{m+2} \mid 0=x_{\hat{0}}\leq x_{a_1} \leq
  x_{a_2}\leq \cdots \leq x_{a_m}\leq x_{\hat{1}}=1\}. \end{equation} 
In this section, we
show that given a planar graph $G$, the canonical triangulation of
$\widehat{\O}(P_G)$ maps to  a  DKK triangulation of  $\F_G$ via the integral equivalence $\phi$ from Theorem \ref{o-f}. This result was first observed by Postnikov \cite{P13}. We also construct a direct bijection between linear extensions of $P_G$, which  index the  canonical triangulation of
$\widehat{\O}(P_G)$, and maximal cliques of $G$, which index the DKK triangulation of $\F_G$.
In  Section \ref{sec:tri}, we   prove for a general graph $G$ that the DKK triangulations of $\F_G$ are  framed Postnikov-Stanley triangulations of $\F_G$. In particular, the canonical triangulation of $\widehat{\O}(P_G)$ for a planar graph $G$  maps to a framed Postnikov-Stanley triangulation of $\F_G$ under  integral equivalence $\phi$ from Theorem \ref{o-f}.

In the following subsection we review the results of Danilov, Karzanov and Koshevoy \cite{kosh}.

\subsection{Danilov--Karzanov--Koshevoy triangulation of flow
  polytopes}

Let  $G$ be a connected graph on the vertex set $[n]$ with edges oriented from
smaller to bigger vertices. Recall from
Proposition~\ref{prop:verticesflowpoly}, that vertices of $\F_G$ are
given by unit flows along maximal directed paths from the
source $1$ to the sink $n$. Following \cite{kosh}, we call such
maximal paths 
\textbf{routes}. 

The following definitions  also follow
\cite{kosh}. Let $v$ be an \textbf{inner} vertex of $G$ whenever $v$ is neither a source nor a sink. 
Fix a \textbf{framing} at each inner vertex $v$, that is, a linear
ordering $\prec_{\i(v)}$ on the set of incoming edges $\i(v)$ to $v$
and the linear ordering $\prec_{\out(v)}$ on the set of  outgoing
edges $\out(v)$ from $v$. A {\bf framed graph}, denoted by $(G,\prec)$, is a graph $G$ with a
framing $\prec$ at each inner vertex. For a framed graph $G$ and an inner
vertex $v$, we denote by $\In(v)$ and by $\Out(v)$ the set of maximal
paths ending in $v$ and   the set of maximal paths starting at $v$,
respectively. We define the order $\prec_{\In(v)}$ on the paths in
$\In(v)$ as follows. If $P, Q \in \In(v)$, $P\neq Q$,  then let $w$ be the 
unique vertex after which $P$ and $Q$ coincide and before which they
differ. Let $e_P$ be the edge of $P$ entering $w$ and $e_Q$ be the edge of $Q$ entering $w$. Then $P \prec_{\In(v)} Q$ if and only if $e_P \prec_{\i(w)} e_Q$.  Similarly, if $P, Q \in \Out(v)$, $P\neq Q$,  then let $w$ be the 
unique vertex before which $P$ and $Q$ coincide and after which they
differ. Let $e_P$ be the edge of $P$ leaving $w$ and $e_Q$ be the edge of $Q$ leaving $w$. Then $P \prec_{\Out(v)} Q$ if and only if $e_P \prec_{\out(w)} e_Q$.

Given a route $P$ with an inner vertex $v$, denote by $Pv$ the maximal
subpath of $P$ ending at $v$ and by $vP$ the
maximal subpath of $P$ starting at $v$. We say that the routes $P$ and $Q$ are \textbf{coherent at a vertex} $v$ which is an inner vertex of both $P$ and $Q$ if the paths $Pv, Qv$ are ordered the same way as $vP, vQ$; that is,  
$Pv \prec_{\In(v)}Qv$ if and only if 
$vP \prec_{\Out(v)}vQ$. 
We say that routes $P$ and $Q$ are \textbf{coherent} if they are
coherent at each common inner vertex. We call a set $C$ of mutually
coherent routes a \textbf{clique}. Let 
$\mathcal{C}^{\max}(G,\prec)$ be the set of maximal cliques
(with respect to number of routes) of the framed graph $G$. 

\begin{definition} \label{def:simp} Given a framed graph $G$, and a  clique $C$ of the framed graph $G$, denote by $\Delta_{C}$ the  convex hull of the vertices of $\F_G$ corresponding to the unit flows along  routes in the   clique $C$. \end{definition}

Theorem \ref{k} below is a special case of \cite[Theorems 1 \& 2]{kosh}.

\begin{theorem}  \cite[Theorems 1 \& 2]{kosh} \label{k} Given a framed
  graph $(G,\prec)$,    the set of simplices 
\[
\{\Delta_C \mid C \in
  \mathcal{C}^{\max}(G,\prec)\},
\] 
corresponding to maximal cliques of the framed graph $G$ are the top dimensional simplices in a regular unimodular triangulation of $\F_G$. 
 Moreover, lower dimensional simplices $\Delta_C$ of this triangulation are obtained as convex hulls of the vertices corresponding to the routes in non-maximal cliques $\mathcal{C}$ of $G$.
\end{theorem}

We call the triangulations specified in Theorem \ref{k} the
\textbf{Danilov-Karzanov-Koshevoy (DKK) triangulations of $\F_G$}. Each such triangulation comes from a particular framing of the graph.
 We are now ready to prove that the canonical triangulation of
 $ \widehat{\O}(P_G)$ is integrally equivalent to a DKK triangulation of $\F_G$ via the map $\phi: \widehat{\O}(P_G)\rightarrow \F_G$ from Theorem \ref{o-f}. We
 now define the framing needed for this result. Consider a planar
 graph $G$ on the vertex set $[n]$ with a particular planar embedding  so that if vertex $i$ is
in position $(x_i,y_i)$ then $x_i < x_j$ whenever $i<j$.  At each
 vertex $v\in [2,n-1]$ of $G$ there is a natural order on the edges
 coming from the planar drawing of the graph: order the incoming edges
 as well as the outgoing edges top to bottom in increasing order; by top to bottom we mean that if we put a small enough circle $C$ centered at vertex $i$ so that all incoming and outgoing edges to vertex $i$ intersect the circle, then we order the incoming (and outgoing) edges top to bottom by decreasing $y$ coordinates of their intersection with the circle $C$.   We call this framing the \textbf{planar framing} of $G$, to emphasize that this framing comes from a particular planar embedding of the graph $G$.

\subsection{The canonical triangulation of $\widehat{\O}(P_G)$ is integrally equivalent to  a
  DKK triangulation of $\F_G$} \label{sec:kosh}

We are now ready to state  the main result of this section.

\begin{theorem2*}
Given a planar graph $G$, the canonical triangulation of $ \widehat{\O}(P_G)$ maps to the  Danilov-Karzanov-Koshevoy triangulation of $\F_G$ coming from the planar framing via the integral equivalence map $\phi: \widehat{\O}(P_G)\rightarrow \F_G$ given in Theorem \ref{o-f}. 
\end{theorem2*}

We prove Theorem \ref{thm2} together with Theorem \ref{thm:bij} below.

\medskip

Recall that by Theorem \ref{thm:stop} vertices of ${\O}(P_G)$ are in bijection with order ideals of $P_G$ -- indeed the vertices of ${\O}(P_G)$ are the characteristic functions of the  complements of the order ideals in the poset $P_G$. By Lemma \ref{lem:stan}  the vertices of $\widehat{\O}(P_G)$ are also naturally indexed by the order ideals of  $P_G$.  Let  $f_I$ be the vertex of $\widehat{\O}(P_G)$ indexed by the order ideal
$I$ of $P_G$.  Given a planar graph $G$ we say that a route $R$ of $G$ separates the order ideal $I$ and the complement  $P_G \setminus I$ if the elements of $P_G$ below the route $R$ in the planar drawing of $G$ and the truncated dual $P_G$ are exactly the elements of the order ideal $I$.

\begin{proposition} \label{prop:vert}
Given a vertex $f_I$ of  $\widehat{\O}(P_G)$ indexed by the order ideal
$I$ of $P_G$ we have that  $\phi(f_I)$ is the unit flow along the route $R$ in $G$
separating $I$ and $P_G \setminus I$.  Moreover, any route $R$ in $G$
separates some order ideal $I$ and $P_G \setminus I$.
\end{proposition}

\begin{proof}
As explained in Section \ref{sec:flow},  the elements of $P_G$ correspond to
bounded regions defined by $G$. Given an indicator function $f_I$ for
the complement of an order ideal $I$ of this poset, by Definition \ref{def:phi} the flow $\phi(f_I)$ is the specified unit flow. See
Figure~\ref{linext} for an example.  
\end{proof}

Next, we define the map $\Phi_{\Delta} $  between linear extensions of $P_G$, which index the top dimensional simplices in the canonical triangulations of $\widehat{O}(P_G)$,  and sets of routes corresponding to the vertices of top dimensional simplices in a DKK triangulation of  $\F_G$ (the latter is shown in Theorem \ref{thm:bij}). 

\begin{definition}
Given a linear extension ${\bf a} = a_1 \cdots
a_m$ of $P_G$, let $\Phi_{\Delta}({\bf a})$ be the following set of routes of
$G$ determined by the order ideals whose elements are the letters in  the prefixes of ${\bf a}$,
\[
\Phi_{\Delta}({\bf a}) \, := \, \{  \phi( f_{\{a_1,\cdots, a_k\}} ) \mid
k=0,1,\ldots,m\}.
\]
\end{definition}

That is, $\Phi_{\Delta}({\bf a})$ is the set of routes of $G$ separating each
of the order ideals formed from letters in the prefixes of the linear extension
$a_1,\ldots,a_m$. See Figure
\ref{linext} for an example.

\begin{figure}[hbtp]
\includegraphics[scale=0.8]{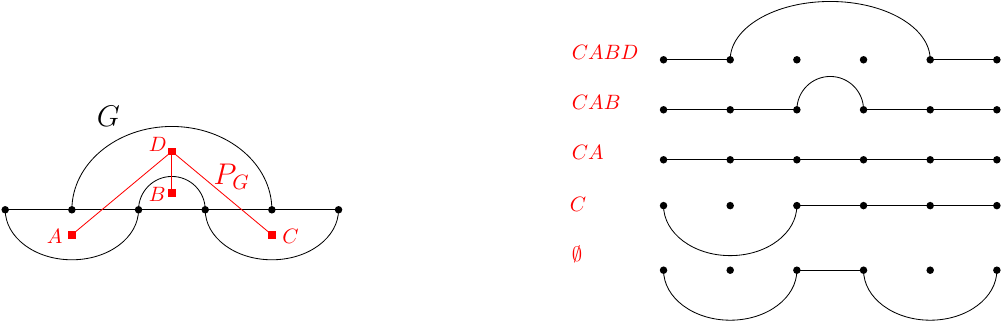}
\caption{On the left is the planar graph $G$ and the poset $P_G$
on elements $A, B, C, D$. On the right are all prefixes of the linear extension $CABD$ of $P_G$ -- each of which corresponds to an order ideal of $P_G$ --, which specify the $0$ coordinates of the vertices of $\widehat{\O}(P_G)$, and the routes these vertices  correspond to under the map $\phi$. Note that the five resulting routes form a maximal clique in $G$ with respect to the planar framing that orders both the incoming and outgoing edges top to bottom in increasing order.  Under $\Phi_{\Delta}$   the   linear extension $CABD$ is mapped to the  maximal clique in $G$ formed by the five routes on the right.}
\label{linext}
\end{figure}

\medskip

Next, we show that the routes in $\Phi_{\Delta}({\bf a})$ form a clique.

\begin{lemma} \label{lem:coh} For a planar graph $G$, fix a linear
  extension ${\bf a} = a_1 \cdots a_m$ of $P_G$ indexing a simplex  $\widehat{\Delta}_{a_1
  \cdots a_m}$ of $\widehat{O}(P_G)$. Then the routes in
  $\Phi_{\Delta}({\bf a})$ are pairwise coherent  in the planar framing of $G$. 
\end{lemma}

\proof Let $v_1$ and $v_2$ be vertices of $\widehat{\Delta}_{a_1
  \cdots a_m}$ mapping to routes $P_1$ and $P_2$ under $\phi_{v}$. It
suffices to show that $P_1$ and $P_2$ are coherent in the planar framing of $G$.

Let the coordinates of $v_1$ equal to $0$ be $x_{\hat{0}}, x_{a_1}, \ldots, x_{a_{k_1}}$ and the coordinates of $v_2$ equal to $0$ be $x_{\hat{0}}, x_{a_1}, \ldots, x_{a_{k_2}}$ and assume without loss of generality that $k_1<k_2$. Since both ${a_1}, \ldots, {a_{k_1}}$ and ${a_1}, \ldots, {a_{k_2}}$ are prefixes of the linear extension of ${a_1}, \ldots, {a_{m}}$, we see that the upper boundary of the regions corresponding to  ${a_1}, \ldots, {a_{k_1}}$ lies weakly below that of the boundary of the regions corresponding to ${a_1}, \ldots, {a_{k_2}}$, and thereby the corresponding routes $P_1$ and $P_2$ are coherent with respect to the planar framing. \qed

\begin{theorem} \label{thm:bij} Given a planar graph $G$, the map $\Phi_{\Delta}$ defined above is a bijection between linear extensions of $P_G$ and maximal cliques in $G$ in the planar framing.
\end{theorem}

\begin{proof}[Proof of Theorems \ref{thm2} \&  \ref{thm:bij}.] 
By Theorem \ref{o-f}, $\phi$ is an integral equivalence between  $\widehat{\O}(P_G)$ and $\F_G$. In particular, the polytopes $\widehat{\O}(P_G)$ and $\F_G$ are of the same dimension and same relative volume. Therefore, the top dimensional simplices in their respective triangulations have the same number of vertices, and  the number of  simplices in any of their unimodular triangulations are the same. Thus, to show Theorems \ref{thm2} \& \ref{thm:bij} it suffices to show that   $\phi$ restricts to a bijection on the vertices of  $\widehat{\O}(P_G)$ and $\F_G$ and that the set of routes that $\Phi_{\Delta}$ maps a linear extension to are pairwise coherent in the planar framing. The former  is follows from Proposition \ref{prop:vert}, while the latter from Lemma \ref{lem:coh}.  
\end{proof}

\begin{corollary}
Given a planar graph $G$,  the number of
linear extensions of $P_G$ equals the number of maximal cliques in $G$ in any framing.
\end{corollary}

\proof The statement is immediate from Theorem \ref{thm2} for the planar framing. But since the Danilov-Karzanov-Koshevoy triangulations are unimodular,   the number of maximal cliques in $G$ is independent of the framing. \qed

\section{Triangulations of  flow   polytopes of general graphs}
\label{sec:tri}

In Theorem \ref{thm:bij} we gave a bijection from linear extensions of
$P_G$ to maximal cliques of $G$ in the planar framing. In  this
section we will see that given any two framings of a graph $G$ (not necessarily
planar) there is a natural bijection between their sets of
maximal cliques. Therefore, combining the bijection from Theorem
\ref{thm:bij} and the one just mentioned, we obtain a bijection
between linear extensions of $P_G$ and maximal cliques in any framing
of a planar graph $G$.

More generally, this section is devoted to studying the set of 
DKK triangulations of a flow polytope $\F_G$ and the framed
Postnikov-Stanley (PS) triangulations of $\F_G$, which we define in this section. 
We show that the set of 
DKK triangulations of a flow polytope $\F_G$ is equal  the set of framed PS triangulations of $\F_G$. 
As a consequence of our proof, we obtain a bijection between the objects indexing 
the PS triangulation of a flow polytope $\F_G$, namely,  nonnegative
integer flows on the graph $G$ with netflow $(0, d_2, \ldots, d_{n-1},
-\sum_{i=2}^{n-1} d_i)$, where $d_i$ is the indegree of vertex $i$ in
$G$ minus $1$\footnote{See Definition \ref{def:flow} and the
  discussion in Section \ref{subsec:31} for the relation of
  nonnegative integer flows with a given netflow vector to Kostant
  partition functions as well as Theorem \ref{ps}.}, and the
objects indexing the  DKK triangulation of a flow polytope $\F_G$,
namely,  maximal cliques in a fixed framing of $G$.  This answers
Postnikov's question \cite{P13} about a bijection between the  sets
indexing the maximal simplices of both triangulations. We also obtain a natural bijection between the sets of maximal cliques of $G$ in different framings, as mentioned in the previous paragraph.

\subsection{Framed Postnikov-Stanley triangulations} \label{sec:ps} We now define framed Postnikov-Stanley triangulations. 
These triangulations were used in \cite{MM}, though they were not
described explicitly there, and we follow closely the exposition
therein. 

A {\bf bipartite noncrossing tree} is a  tree with 
 {left vertices}
$x_1,\ldots,x_{\ell}$ and {right vertices $x_{\ell+1},\ldots,
  x_{\ell+r}$} with no pair of edges $(x_p,x_{\ell+q}),
(x_t,x_{\ell+u})$ where $p<t$ and $q>u$. 
 We denote by
$\mathcal{T}_{\I,\O}$ the set of  bipartite noncrossing trees
where $\I$ and $\O$ are the ordered
sets $(x_1,\ldots,x_{\ell})$ and $(x_{\ell+1},\ldots,x_{\ell+r})$ respectively.  We have that $\#
\mathcal{T}_{\I,\O}=\binom{\ell+r-2}{\ell-1}$, since the elements of  $\mathcal{T}_{\I,\O}$  are in
bijection with weak compositions of $\ell-1$ into $r$ parts. A
tree $T$ in $\mathcal{T}_{\I,\O}$ corresponds to the composition
$(b_1,\ldots,b_r)$ of $(\text{indegrees $-1$})$, where $b_i$ denotes
the number of edges incident to the right vertex $x_{\ell+i}$ in $T$
minus ~$1$. 

\begin{example}
The bipartite tree in Figure~\ref{ordering} corresponds to the
composition $(1,0,2)$.
\end{example}

We now define what we mean by a \text{reduction} at vertex $i$ of a
framed graph $G$ on the vertex set $[n]$.  Let $\I_i$ denote the
multiset of incoming edges and $\O_i$ the multiset of outgoing edges of $i$. In addition, we assume that $\I_i$ and $\O_i$ are linearly ordered according to the framing of $G$.  A reduction performed at $i$ of $G$ results in several new graphs indexed by  bipartite noncrossing trees on the left vertex set $I_i$ and right  vertex set $\O_i$. We define these new graphs precisely below.

Consider a tree $T \in \mathcal{T}_{\mathcal{I}_i,\mathcal{O}_i}$. For each tree-edge $(e_1,e_2)$ of $T$ where $e_1=(r,i) \in \mathcal{I}_i$ and $e_2=(i,s) \in \mathcal{O}_i$, let $e_1+e_2$ be the following edge:
\begin{equation}\label{treeEdge}
e_1+e_2=(r,s).\end{equation} 
We call the edge $e_1+e_2$ the \textbf{sum of  edges}. Alternatively
you can consider this as a path in $G$ consisting of edges $e_1$ and $e_2$. Inductively, we can also define the sum of more than two consecutive edges. 

Given $T$ in $\mathcal{T}_{\mathcal{I}_i,\mathcal{O}_i}$, let
$G^{(i)}_T$ be the graph obtained from $G$ by removing the vertex $i$
and all the edges of $G$ incident to $i$ and adding  the multiset of
edges $\{\{e_1+e_2 ~|~ (e_1,e_2)\in E(T)\}\}$. See Figures \ref{G_T}, \ref{fig:subdivhrA} and \ref{fig:big} for examples of $G^{(i)}_T$.

Given a tree $T$ in $\mathcal{T}_{\I_i, \O_i}$, a \textbf{reduction of
  $G$ at the vertex $i$  with respect to $T$}  replaces $G$ by the
graphs in $G^{(i)}_T$ defined
above. The reduction also keeps track which sum of the edges of $G$ is  each edge of the new graphs (allowing for the sum of only one element, when an edge was left intact). 

\begin{figure}
\subfigure[]{
\includegraphics[width=7cm]{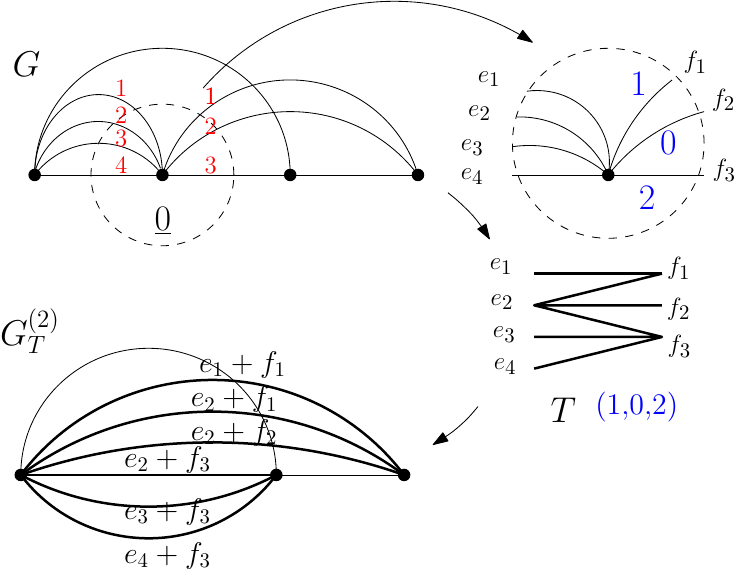}
}
\quad 
\subfigure[]{
\includegraphics[width=7cm]{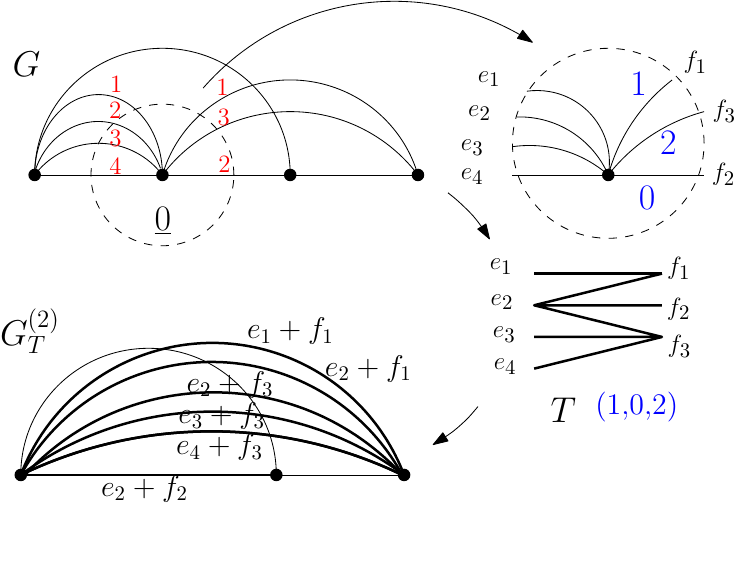}
}
\caption{Replacing the incident edges of vertex $2$ in  a graph $H$,
  by a noncrossing tree $T$ encoded by the composition $(1,0,2)$ of
  $3=indeg_H(2)-1$ using two different framings (indicated by the blue
  numbers incident to vertex $2$ in $G$): (a) the framing is increasing top to
  bottom, (b) different framing.}
\label{G_T}
\end{figure}

\begin{figure}
\includegraphics[height=2cm]{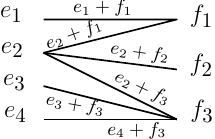}
\caption{Based on the noncrossing tree: $S(f_1)=\{e_1+f_1,
  e_2+f_1\},$ $S(f_2)=\{e_2+f_2\}$, and
  $S(f_3)=\{e_2+f_3, e_3+f_3, e_4+f_3\}$. The local orderings of these edges at the vertices to which they are incoming are $e_1+f_1<e_2+f_1$ and $e_2+f_3<e_3+f_3<e_4+f_3$.}
\label{ordering}
\end{figure}

We now define an  \textbf{inheritance framing} of $G^{(i)}_T$ for $T$
in $\mathcal{T}_{\I_i,\O_i}$, which it inherits from the framing of
$G$ as follows:

\begin{itemize}
\item[(i)] The edges incident to a vertex $j$ smaller than $i$ in $G^{(i)}$ are
in bijection with edges incident to vertex $j$ in $G$. We order the edges in $G^{(i)}$ in the same way as they are ordered in $G$. 
\item[(ii)] For each vertex $j$ greater than $i$ the  multiset of outgoing edges $\O_j(G^{(i)}_T)$ equals  $\O_j(G)$. We order these the edges of $\O_j(G^{(i)}_T)$ the same way the edges   $\O_j(G)$ are ordered.
\item[(iii)] 
For each vertex $j$ greater than $i$, if $\I_j(G)=\{m_1, \ldots,m_k\}$ (the
 multiset linearly ordered according to the framing of $G$),  then
 the multiset $\I_j(G^{(i)}_T)$ consists of edges that are sums of
 edges of $G$ (potentially the empty sum) with an edge of
 $\I_j(G)$. Thus denote by
 $S(m_l)$, $l\in[k]$, the edges in  $\I_j(G^{(i)}_T)$ which are sums
 of edges of $G$ (potentially the empty sum) with $m_l$.  Then let  any edge in $S(m_p)$ be less than any edge in  $S(m_q)$ for $p<q$, $p, q \in [k]$. We now specify the ordering of the edges within the sets $S(m_l)$, $l \in [k]$.    
 If  $S(m_l)=\{m_l\}$  then there is nothing to specify. If  $S(m_l)\neq \{m_l\}$, then draw $T$ with the left and right sets of vertices ordered
 vertically following the linear order of $\I_i$ and $\O_i$ from the
 framing of $G$. We order the edges in
 $S(m_l)$ following the order on the edges of the noncrossing bipartite tree
 $T$ when viewed from top to bottom (smallest edge to largest).  See  Figure~\ref{ordering} for an example. 
\end{itemize}

Next we describe what we refer to as the  \textbf{framed Postnikov-Stanley (PS)
  triangulations} of $\F_G$. 

Given a framed graph $(G, \prec)$ on the vertex set $[n]$, and a nonnegative integer flow
$\ifl(\cdot)$ on $G$ with netflow $(0,d_2,d_3,\ldots,d_{n-1}, -\sum_i
d_i)$, where $d_i=indeg_i(G)-1$, we explain how to obtain a simplex $\Delta_{\ifl}^{(G, \prec)}$, such that as $\ifl$ runs over all nonnegative integer flow $G$ with netflow $(0,d_2,d_3,\ldots,d_{n-1}, -\sum_i
d_i)$ we obtain a set of simplices $\Delta_{\ifl}^{(G, \prec)}$ that are the top dimensional simplices of a triangulation of $\F_G$. It is this triangulation that we term the framed Postnikov-Stanley (PS)
  triangulation.  

Given a framed graph $(G, \prec)$ on the vertex set $[n]$, and a nonnegative integer flow
$\ifl(\cdot)$ on $G$ with netflow $(0,d_2,d_3,\ldots,d_{n-1}, -\sum_i
d_i)$ we read off the nonnegative integer flow values specified by $\ifl(\cdot)$ on the edges of
$\mathcal{O}_2(G)$ yielding a composition $(c_1, \ldots, c_{\# \mathcal{O}_2(G)})$ of $d_2 =
\#\mathcal{I}_2(G)-1$. Here $c_j$ corresponds to the flow value on the $j$th largest edge in $\mathcal{O}_2(G)$ in the framing. Using this composition $(c_1, \ldots, c_{\# \mathcal{O}_2(G)})$ of  $d_2 =
\#\mathcal{I}_2(G)-1$ , we  build a bipartite tree $T_2$   in
$\mathcal{T}_{\I_2,\O_2}$ as follows. The sets $\I_2$ and $\O_2$ have an ordering in the framing of $G$. Assume that this ordering is $\I_2=\{v_1< \cdots< v_{\# \I_2}\}$ and  $\O_2=\{w_1< \cdots< w_{\# \O_2}\}$. Draw the bipartite tree  $T_2$ in
$\mathcal{T}_{\I_2,\O_2}$  with the left vertex set $\I_2=\{v_1< \cdots< v_{\# \I_2}\}$ so that the vertices $v_1< \cdots< v_{\# \I_2}$ are ordered top to bottom vertically on the left. Similarly, the right vertices $w_1< \cdots< w_{\# \O_2}$ are drawn top to bottom vertically on the right. We let the degree of vertex $w_j$ in $T_2$ be $c_j+1$. The above uniquely determines the noncrossing bipartite tree $T_2$ on left and right vertex sets $\I_2$ and $\O_2$. See Figure   \ref{G_T} for an example. 
  With  tree $T_2$   constructed,  we do a reduction at vertex $2$  to
obtain $G_2 := G_{T_2}^{(2)}$ with an inheritance framing.  

Recursively, given $G_{i-1}$, we read off the integer flow values from $\ifl(\cdot)$ on the edges of
$\mathcal{O}_{i}(G) = \mathcal{O}_{i}(G_{i-1})$. These flow values can be seen as components of a composition $(c_1, \ldots, c_{\# \mathcal{O}_i(G)})$ of 
\[
d_i + \sum_{e, \fin(e)=i} \ifl(e) \,=\,\#\mathcal{I}_i(G_{i-1})-1.
\]
\
The component $c_j$ corresponds to the flow value on the $j$th largest edge in $\mathcal{O}_i(G)=\mathcal{O}_i(G_{i-1})$ in the framing. 
From this composition    $(c_1, \ldots, c_{\# \mathcal{O}_i(G)})$ we build a
bipartite tree $T_i$   in
$\mathcal{T}_{\mathcal{I}_i(G_{i-1}),\mathcal{O}_i(G_{i-1})}$ as follows. The sets $\mathcal{I}_i(G_{i-1})$ and $\mathcal{O}_i(G_{i-1})$ have an ordering in the inheritance framing of $G_{i-1}$. Assume that this ordering is $\mathcal{I}_i(G_{i-1})=\{v_1< \cdots< v_{\# \mathcal{I}_i(G_{i-1})}\}$ and  $\mathcal{O}_i(G_{i-1})=\{w_1< \cdots< w_{\# \mathcal{O}_i(G_{i-1})}\}$. Draw the bipartite tree  $T_i$   in
$\mathcal{T}_{\mathcal{I}_i(G_{i-1}),\mathcal{O}_i(G_{i-1})}$  with the left vertex set $\mathcal{I}_i(G_{i-1})=\{v_1< \cdots< v_{\# \mathcal{I}_i(G_{i-1})}\}$ so that the vertices $v_1< \cdots< v_{\# \mathcal{I}_i(G_{i-1})}$ are ordered top to bottom vertically on the left. Similarly, the right vertices $w_1< \cdots< w_{\# \mathcal{O}_i(G_{i-1})}$ are drawn top to bottom vertically on the right. We let the degree of vertex $w_j$ in $T_i$ be $c_j+1$. The above uniquely determines the noncrossing bipartite tree $T_i$ on left and right vertex sets $\mathcal{I}_i(G_{i-1})$ and $\mathcal{O}_i(G_{i-1})$.   With  tree $T_i$   constructed  we do a reduction at vertex $i$  to
obtain $G_i := (G_{i-1})_{T_i}^{(i)}$ with an inheritance framing.    We iterate this for $i=1,\ldots,n-1$. See Figure \ref{fig:big} for an example.

 \begin{figure}
\begin{center}
\includegraphics[scale=.9]{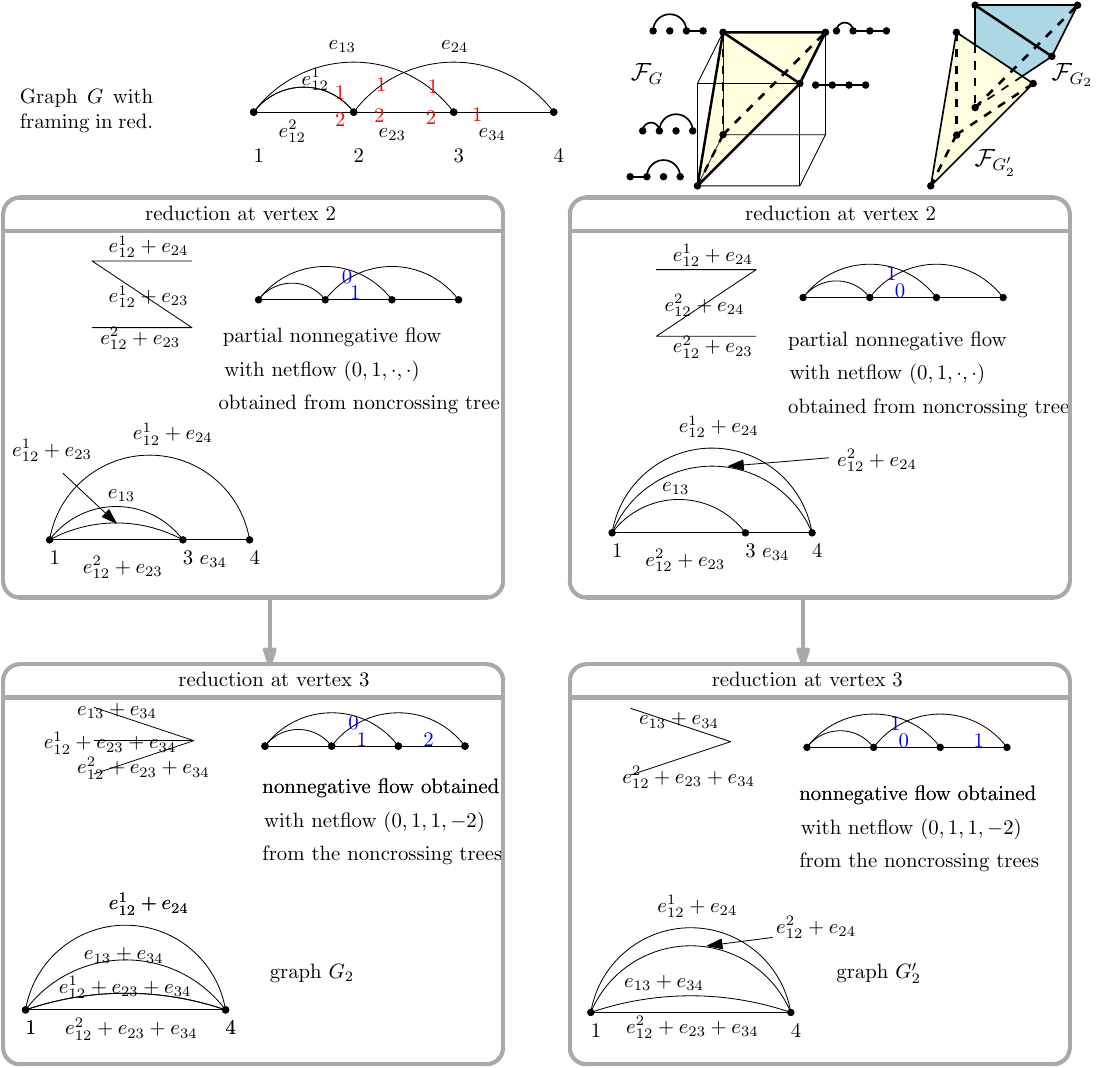}
\end{center}
\caption{Reductions executed at vertex $2$ and $3$ of the framed graph $G$. Noncrossing trees encoding the reduction are displayed with all edges labeled. The nonnegative flow on $G$ with netflow $(0,1,1,-2)$ is built. The flow polytope $\F_G$ is dissected into two simplices corresponding to $G_2$ and $G_2'$.}
\label{fig:big}
\end{figure}

Thus, from the integer flow $\ifl(\cdot)$ we obtain a tuple of bipartite
noncrossing trees $(T_2,T_3,\ldots,T_{n-1})$ such that $G_i :=
(G_{i-1})_{T_i}^{(i)}$ for $i=2,\ldots,n-1$ and $G_1 := G$. Since $G_{n-1}$
has no incoming or outgoing edges to vertices $i=2,\ldots,n-1$ then
$G_{n-1}$ consists of two vertices $1$ and $n$ and $\#E(G)-n+2$ multiple
edges.  Thus $\mathcal{F}_{G_{n-1}}$ is  a
$(\#(G)-n+1)$-simplex. 

Recall that each
such multiple edge $e$ in $\mathcal{F}_{G_{n-1}}$ is a
sum of edges of the original graph of $G$ as explained in the beginning of this section. Such sum of edges
corresponds to a route in the graph  $G$, i.e. a directed path from
vertex $1$ and $n$ in $G$. We denote the unit flow in $G$ along the route corresponding to edge $e$
in $G_{n-1}$ by $\rho(e)$ and we let 

$$\Delta_{\ifl}^{(G,\prec)}:={\rm ConvHull}\{\rho(e) \mid e \in E({G_{n-1}})\}$$ be the simplex
with vertices $\rho(e)$. Note that
$\Delta_{\ifl}^{(G,\prec)}$ is integrally equivalent to $\mathcal{F}_{G_{n-1}}$, and it is a subset of $\mathcal{F}_{G}$.

Postnikov and Stanley proved
Theorem~\ref{ps} by showing that this iterative construction of
simplices from integer flows yields a triangulation of
$\mathcal{F}_G$. Denote by $\F_G^{\integer}(0,d_2,\ldots,d_{n-1},-\sum_i d_i)$ \textbf{ the set of nonnegative integer flows on the graph $G$ with netflow} $(0,d_2,\ldots,d_{n-1},-\sum_i d_i)$.

\begin{theorem}[{cf. \cite[\S 6.1]{MM}}] \label{psflows}
Given a framed graph $(G,\prec)$, the set of simplices 
\[
\{\Delta_{\ifl}^{(G,\prec)} \mid
\ifl \in \F_G^{\integer}(0,d_2,\ldots,d_{n-1},-\sum_i d_i) \},
\]
 where $d_i=indeg_i(G)-1$, are the top simplices of a unimodular triangulation of $\F_G$.  
\end{theorem}

\begin{remark} \label{rem:more}
Note that the triangulation in \cite[\S
6.1]{MM} comes from the top to bottom framing of the graph.  Theorem \ref{psflows}  yields a triangulation for any  framing of the graph. The proof in \cite[\S
6.1]{MM} adapts readily for an arbitrary framing. Indeed,   more general triangulations can be constructed in the above way that do not depend on a fixed framing of the graph $G$; we only need to specify some (any) ordering of edges at each vertex as we do the reductions. 
\end{remark}

\begin{figure}
\includegraphics{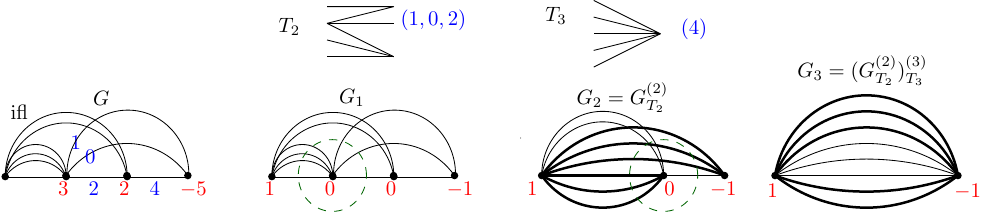}
\caption{Example in the Postnikov--Stanley triangulation of
$\F_G$ of how to find a
simpex $\F_{G_3}$ from an integer flow $\ifl$ in $\F_{G}(0,d_2,d_3,-d_2-d_3)$ where $d_i=indeg_i(G)-1$. Each step of the subdivision is encoded by noncrossing
trees $T_{i+1}$ that are equivalent to compositions
${(b_1,\ldots,b_r)}$ of $\#\mathcal{I}_{i+1}(G_i)-1$ with
$\#\mathcal{O}_{i+1}(G_i)$  parts. These trees or compositions are
read from the integer flow. The framing used is top to bottom.}
\label{fig:subdivhrA}
\end{figure}

 \subsection{The set of DKK triangulations equals the set of  framed PS triangulations} \label{ss} 
In this section we show that with a fixed framing $(G, \prec)$ the  DKK triangulations  and the PS triangulation are identical. In effect, the set of DKK triangulations equals the set of  framed PS triangulations. We also give an 
 explicit bijection between the objects
indexing a DKK triangulation of $\F_G$ for a framing of $G$
and a framed PS triangulation of $\F_G$, namely a bijection between
maximal cliques of $G$ with respect to a fixed framing and nonnegative integer flows of $G$ with netflow $(0,d_2,\ldots,d_{n-1},-\sum_i d_i)$.

The following results show that the vertices of a simplex
$\Delta_{\ifl}^{(G, \prec)}$ correspond to a maximal clique of the framed graph $(G, \prec)$. 
Recall that the simplex $\Delta_{\ifl}^{(G, \prec)}$ is integrally equivalent to the flow
polytope $\F_{G_{n-1}}$ of a graph $G_{n-1}$ consisting of vertices
$1$ and $n$ and $\#E(G)-n+2$ multiple edges $(1,n)$ and that the set of simplices $\Delta_{\ifl}^{(G, \prec)}$, as $\ifl$ runs over all flows in $\F_{G}^{\integer}(0,d_2,d_3,d_4,d_5,
-\sum_i d_i)$ forms the top dimensional simplices of a unimodular triangulation of $\F_G$ as shown in Theorem \ref{psflows}.

\begin{proposition} \label{dkk-ps} Given a framed graph $(G, \prec)$ with
  vertices $[n]$ and a nonnegative integer flow $\ifl \in \F_{G}^{\integer}(0,d_2,d_3,d_4,d_5,
-\sum_i d_i)$, where $d_i=indeg_i(G)-1$,  the 
  routes of $G$ along which the unit flows  give the vertices of  the simplex $\Delta_{\ifl}^{(G, \prec)}$ form a maximal clique with respect to the
  coherence relation in $(G, \prec)$. 
\end{proposition}
 
\begin{proof} Recall that $\Delta_{\ifl}^{(G, \prec)} \intequiv \F_{G_{n-1}}$,  for some  $G_{n-1}$ as described in Section \ref{sec:ps}.  Recall that a sequence of graphs $G_1:=G, G_2, \ldots, G_{n-1}$ encode the successive reductions leading to the simplex   $\Delta_{\ifl}^{(G, \prec)} \intequiv \F_{G_{n-1}}$. Graph $G$ has a framing, and the framing of  graph $G_i$, $i \in [2,n-1]$, is the inheritance framing obtained from the framing of $G_{i-1}$.
Suppose that to the contrary, there are two vertices of the simplex
$\Delta_{\ifl}^{(G, \prec)}$, which correspond to non-coherent routes $P$ and $Q$ in $G$. Suppose that $P$ and $Q$ are not coherent at the common inner vertex $v$. Suppose that the smallest vertex after which $Pv$ and $Qv$ agree is $w_1$ and the largest vertex before which $vP$ and $vQ$ agree is $w_2$. Let the edges incoming to $w_1$ be $e_P^1$ and $e_Q^1$ for $P$ and $Q$, respectively, and let  the edges outgoing from $w_2$ be $e_P^2$ and $e_Q^2$ for $P$ and $Q$, respectively. Since $P$ and $Q$ are not coherent at $v$, this implies that either $e_P^1\prec_{\i(w_1)} e_Q^1$ and $e_Q^2\prec_{\out(w_2)} e_P^2$ or $e_Q^1\prec_{\i(w_1)} e_P^1$ and $e_P^2\prec_{\out(w_2)} e_Q^2$. We also have that the segments of $P$ and $Q$ between $w_1$ and $w_2$ coincide.

Denote by $p$ the sum of edges between $w_1$ and $w_2$ on $P$. Denote
by $*(e_{{Z}}^1+p)$, for $Z\in \{P, Q\}$, the sum of edges left of
$w_2$  that are edges in ${Z}$  (including $e_{{Z}}^1$ in particular). 
 After a
certain number of reductions executed according to the framing, we are
about to perform the reduction at vertex $w_2$. This reduction involves
deleting $w_2$ and the edges incident to it, and adding the edges
obtained from the noncrossing tree $T$ we constructed based on the
ordering of the incoming and outgoing edges at $w_2$. 
In such a noncrossing tree, the vertex corresponding to the edge
stemming from $*(e_{{Z}}^1+p)$,  $Z\in \{P, Q\}$, is above the vertex
$*(e_{{\overline{Z}}}^1+p)$, where $\overline{Z}$ is the complement of
$Z$ in $\{P,Q\}$, in the left
partition of the vertices of $T$. On the other hand, the vertex corresponding
to  $e_{\overline{Z}}^2$ is above the vertex corresponding to  $e_Z^2$
in the right  partition of the vertices of $T$. Thus,  it is
impossible to obtain both routes $P$ and $Q$ as vertices of
$\F_{G_{n-1}}$
since that would force connecting  $*(e_{{Z}}^1+p)$ and $e_{{Z}}^2$ as
well as $*(e_{\overline{Z}}^1+p)$ and $e_{\overline{Z}}^2$ in
$T$. This would make a crossing in the noncrossing tree $T$, a contradiction.
 \end{proof}

Proposition~\ref{dkk-ps} justifies the following definition. 

\begin{definition}
Given a framed graph $(G,\prec)$ on the vertex set  $[n]$, let $$\Lambda^{(G,\prec)}:
  \mathcal{F}^{\integer}_G(0,d_2,\ldots,d_{n-1},\sum_i d_i) \to
  \mathcal{C}^{\max}(G,\prec)$$ be the map defined by
  $\Lambda^{(G,\prec)}(\ifl)=C$, where the
  vertices of the simplex $\Delta_{\ifl}^{(G,\prec)}$ are the unit flows along routes in the
  maximal clique $C$ and $d_i=indeg_i(G)-1$. 
\end{definition}

\begin{example}
Figure~\ref{fig:routes} gives an example of the bijection $\Lambda^{(G,\prec)}$ between the
two integer flows in $\F^{\integer}_G(0,d_2,d_3,-d_2-d_3)$ and the two
maximal cliques with respect to the framing of $G$ given in Figure~\ref{fig:big}.
\end{example}

\begin{example} \label{ex:bijection_flow2clique}
Figure~\ref{fig:example_lambda} gives a larger example of the bijection
$\Lambda^{(G,\prec)}$ between an integer flow in
$\F^{\integer}_{K_6}(0,0,1,2,-3)$ and a maximal clique of $K_6$.
\end{example}

\begin{figure}
\includegraphics{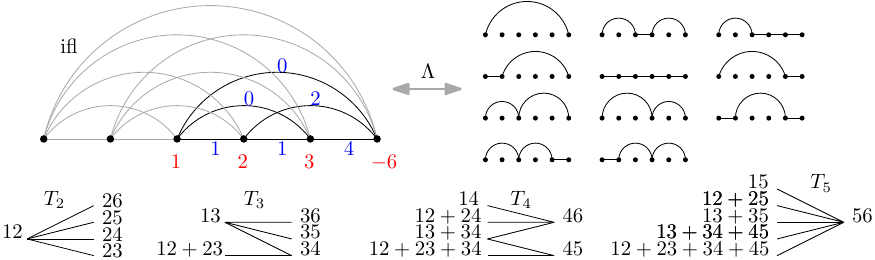}
 \caption{
Example of the bijection $\Lambda = \Lambda^{(G,\prec)}$ between an nonnegative integer flow $\ifl$ in $\F_{G}(0,d_2,d_3,d_4,d_5,
-\sum_i d_i)$ where $G=K_6$ and $d_i=indeg_i(G)-1$. Below are the noncrossing trees $T_i$ with left vertices
$\mathcal{I}_i(G_{i-1})$ and right vertices $\mathcal{O}_i(G_{i-1})$
written as sums of edges of $G$ ($ij$ is shorthand for the edge $(i,j)$). The
framing used is top to bottom.}
\label{fig:example_lambda}
\end{figure}

 \begin{figure}
\begin{center}
\includegraphics[scale=.9]{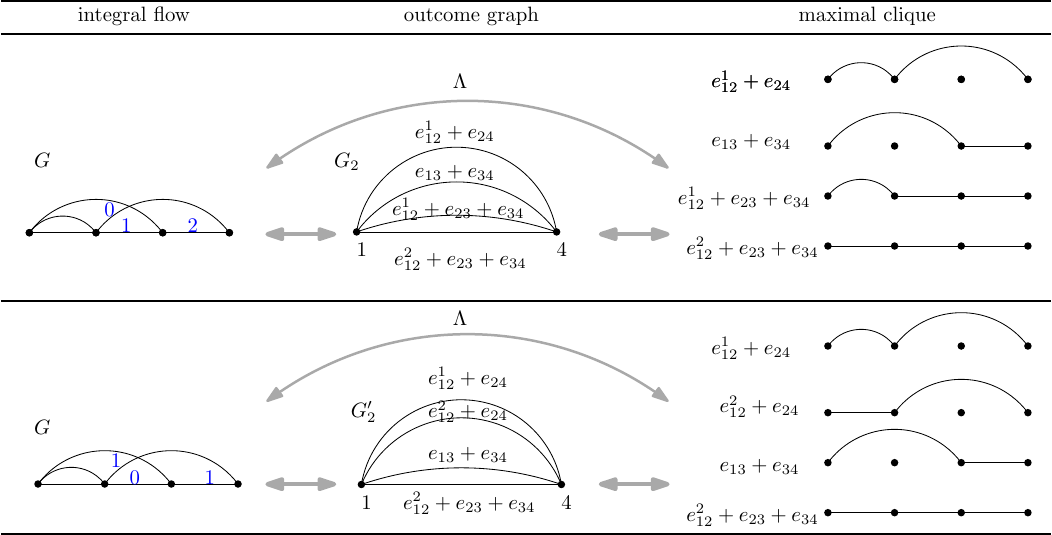}
\end{center}
\caption{ The graphs $G$,  $G_2$ and  $G'_2$  as well as the edge labels $e_{ij}$ are of the edges of the graph $G$   are as in Figure \ref{fig:big}. The four paths on the top correspond to the vertices of the
  simplex given by $G_2$.  The four paths
  on the bottom correspond to the vertices of the simplex given by
  $G'_2$. Both sets of paths are coherent
  in the top to bottom framing of $G$ given in  Figure
  \ref{fig:big}.}
\label{fig:routes}
\end{figure}

We now have:

 \begin{theorem} \label{thm:subset}  Given a framed graph $(G,\prec)$ on the vertex set  $[n]$
the   Danilov-Karzanov-Koshevoy
 triangulations of $\F_G$ with respect to this framing is  the framed Postnikov-Stanley triangulations of $\F_G$ with respect to the same framing. Moreover,     the map
$\Lambda^{(G,\prec)}$ defined above is a bijection between nonnegative  integer flows in
$\mathcal{F}_G^{\integer}(0,d_2,\ldots,d_{n-1},-\sum_i d_i)$, where $d_i=indeg_i(G)-1$,   and
maximal cliques in $\mathcal{C}^{\max}(G,\prec)$. 
 \end{theorem}

\proof  Fix a framing $(G,\prec)$. Proposition
 \ref{dkk-ps} shows that the framed PS triangulation with respect to this framing is the same as the DKK triangulation  with respect to this framing. Therefore, any DKK triangulation is a framed PS triangulation. In particular,  $\Lambda^{(G,\prec)}$ is a bijection that simply sends one set of labelings of a fixed triangulation of $\F_G$ into another set of labelings of the very same triangulation of $\F_G$. \qed
  
  \medskip

We conclude by noting that there is a nice way to describe the inverse
of the map $\Lambda^{(G,\prec)}$:

\begin{lemma} \label{lemma:inverse_bijecion}
Fix  a framed graph $(G,\prec)$ and a   flow
$\ifl \in \F_G^{\integer}(0,d_2,\ldots,d_{n-1},-\sum_i d_i)$, where $d_i=indeg_i(G)-1$. If   
$\Lambda^{(G,\prec)}(\ifl)=C$,  then each edge $e$ of the graph $G$ appears $\ifl(e)+1$
times as an edge of one of the paths ending in $v={\fin(e)}$ in the set (not multiset!) $\{Pv  \mid
 P \in C, v={\fin(e)}\}$ of prefixes of routes in the clique $C$. In particular, given a maximal clique $C \in \mathcal{C}^{\max}(G, \prec)$ the inverse $(\Lambda^{(G,\prec)})^{-1}(C)$ is given by \[
((\Lambda^{(G,\prec)})^{-1}(C))(e)  =  n(e) -1,  
\] where    $n(e)$ is the number of times edge $e$ appears in set
of prefixes $\{Pv  \mid
 P \in C, v={\fin(e)}\}$. \end{lemma}

 \begin{proof}
By the construction of $\Lambda^{(G,\prec)}$, from the integer flow
$\ifl(\cdot)$ we obtain a tuple of noncrossing bipartite trees
$(T_2,T_3,\ldots,T_{n-1})$ such that $G_1=G$ and $G_i :=
(G_{i-1})^{(i)}_{T_i}$ for $i=2,\ldots,n-1$ where $G_{n-1}$ is a graph
with vertices $1$ and $n$ and $\#E(G)-n+2$ multiple edges where each
multiple edge is a sum of edges of the original graph $G$ defining a
route of the maximal clique $C$. 

The edges of intermediate graphs $G_2,\ldots,G_{n-2}$ for
$i=2,\ldots,n-2$ encode prefixes of the routes in the clique $C$ as follows: for the edge $e=(i,j)$ in $\mathcal{O}_i(G)=\mathcal{O}_i(G_{i-1})$, the tree
$T_i$ in $\mathcal{T}_{\mathcal{I}_i(G_{i-1}),\mathcal{O}_i(G_{i-1})}$
has  $\ifl(e)+1$ tree-edges incident  to $e$ by definition. Therefore,   the edge $e$  appears exactly $\ifl(e)+1$ times in the set of 
prefixes of the routes $\{Pv  \mid
 P \in C, v={\fin(e)}\}$.  The statement about $(\Lambda^{(G,\prec)})^{-1}(C)$ then follows readily.
\end{proof}

\begin{example}
We continue with Example~\ref{ex:bijection_flow2clique} illustrated in
Figure~\ref{fig:example_lambda}. Edge $e=(3,4)$ has flow $\ifl(e)=1$
and there are two paths ending in vertex $4$ containing $e$ in the corresponding
clique $C:=\Lambda^{(G,\prec)}(\ifl)$, namely the paths consisting of edges
$(1,2), (2,3), (3,4)$ and of edges $(1,3), (3,4)$. Note that the path $(1,3), (3,4)$ is the prefix of two routes in the clique, however, we count it here just once since in Lemma \ref{lemma:inverse_bijecion} we are looking at the set of prefixes of the routes in the clique and not a multiset of prefixes.
\end{example}

\section*{Acknowledgments} The authors are grateful to Alexander Postnikov for generously 
sharing his insights and questions. The authors  are also grateful to  the anonymus referee for
numerous helpful comments and suggestions.  AHM and JS would like to thank
ICERM and the organizers of its Spring 2013 program in {\em
  Automorphic Forms} during which part of this work was done. The authors also thank the SageMath community~\cite{Sage} for developing and sharing their code by which some of this research was conducted.

\end{document}